\documentclass[11pt,reqno]{amsproc}
\usepackage[margin=1in]{geometry}
\usepackage{graphicx,url,cite,amsmath,amssymb,amsfonts,color}
\usepackage{scalerel}
\usepackage{mathtools}
\usepackage{etoolbox}

\usepackage[dvipsnames]{xcolor}

\newtheorem{thm}{Theorem}
\newtheorem{prop}{Proposition}

\newtheorem{corollary}{Corollary}
\newtheorem{example}{Example}

\newtheorem{defin}{Definition}
\theoremstyle{definition}

\makeatletter
\patchcmd\@thm
  {\let\thm@indent\indent}{\let\thm@indent\noindent}%
  {}{}
\makeatother

\usepackage{etoolbox}
\expandafter\patchcmd\csname\string\proof\endcsname
  {\normalparindent}{0pt }{}{}

\newcommand{\cA}{{\mathcal A}}

\newcommand{\bx}{{ \bold{x} }}

\newcommand{\bX}{{ \bold{X} }}
\newcommand{\by}{{\bold y}}
\newcommand{\tby}{{\tilde{\by}}}
\newcommand{\tz}{{\tilde{z}}}
\newcommand{\tilt}{{\tilde{t}}}
\newcommand{\ttau}{{\tilde{\tau}}}
\newcommand{\tZ}{{\tilde{Z}}}
\newcommand{\bY}{{\bold Y}}
\newcommand{\tbY}{{ \tilde{\bY}}}

\newcommand{\bFF}{{\bold F}}
\newcommand{\bG}{{\bold G}}
\newcommand{\bff}{{\bold f}}

\newcommand{\bzero}{{\mathbf{0}}}
\newcommand{\ve}{{\varepsilon}}

\newcommand{\tvf}{{\tilde{\varphi}}}

\newcommand{\red}{\textcolor{red}} 

\definecolor{mygray}{gray}{0.6}

\newcommand{\be}{\begin{equation}}
\newcommand{\ee}{\end{equation}}
\newcommand{\bea}{\begin{eqnarray}}
\newcommand{\eea}{\end{eqnarray}}

\title[``Life after death" in singular ODEs]{``Life after death" in ordinary differential equations\\ with a non-Lipschitz singularity}
\author{Theodore D. Drivas}
\address{Department of Mathematics, Princeton University, Princeton, NJ 08544}
\email{tdrivas@math.princeton.edu}
\author{Alexei A. Mailybaev}
\address{ Instituto Nacional de Matem\'{a}tica Pura e Aplicada -- IMPA, Rio de Janeiro, Brazil}
\email{alexei@impa.br}

\date{today}
\begin{document}

\begin{abstract}
We consider a class of ordinary differential equations in $d$-dimensions featuring a non-Lipschitz singularity at the origin. Solutions of such systems exist globally and are unique up until the first time they hit the origin, $t = t_b$, which we term `blowup'. However, infinitely many solutions may exist for longer times. To study continuation past blowup, we introduce physically motivated regularizations: they consist of smoothing the vector field in a $\nu$--ball around the origin and then removing the regularization in the limit $\nu\to 0$. We  show that this limit can be understood using a certain autonomous dynamical system obtained by a solution-dependent renormalization procedure. This procedure maps the pre-blowup dynamics, $t < t_b$, to the solution ending at infinitely large renormalized time. In particular, the asymptotic behavior as $t \nearrow t_b$ is described by an attractor. The post-blowup dynamics, $t > t_b$, is mapped to a different renormalized solution starting infinitely far in the past. Consequently, it is associated with another attractor. The $\nu$-regularization establishes a relation between these two different ``lives'' of the renormalized system. We prove that, in some generic situations, this procedure selects a unique global solution (or a family of solutions), which does not depend on the details of the regularization. We provide concrete examples and argue that these situations are qualitatively similar to post-blowup scenarios observed in infinite-dimensional models of turbulence.  
\end{abstract}


\maketitle

\section{Introduction}

Non-Lipschitz singularities in differential equations arise naturally in numerous physical applications. For example, such singularities are the collision points in the N-body problem~\cite{diacu1992singularities}. {Non-Lipschitz singularities, which are reached in finite time, are} wide spread in partial differential equations, where they are often termed as blowup~\cite{eggers2009role}. A classical example is shock formation in conservation laws and ideal compressible fluid systems, where the velocity fields dynamically form jump discontinuities \cite{dafermos2010hyperbolic}.  Another important example arises in the setting of incompressible fluid dynamics: the celebrated Kolmogorov (K41) theory stipulates that an ``ideal" turbulent flow field is only 1/3--H\"{o}lder continuous~\cite{kolmogorov1941local,frisch1995turbulence}. 

A basic understanding of such non-smooth systems can be obtained from the textbook example
\be\label{1dmodel}
{\dot{x} = x^{1/3},} \qquad x(t_0)= x_0.
\ee
{This is a particular case of the equation $\dot{x} = {\rm sgn}(x)|x|^\alpha$ for $\alpha = 1/3$.} The exact solution of \eqref{1dmodel} for $t \ge t_0$ is easily obtained by separation of variables as
\be\label{exactSol}
x(t) = x_0 \left[1 + \frac{2}{3} |x_0|^{-2/3}(t-t_0)\right]^{3/2},\qquad x_0 \ne 0,
\ee
which is the unique solution for $x_0 \ne 0$. The right-hand side of (\ref{1dmodel}) is not Lipschitz continuous at $x_0 = 0$ and, as a result, the solution is non-unique: 
\be\label{exactSolb}
x(t) = \pm \left[\frac{2}{3} (t-t_0)\right]^{3/2}, \qquad x_0 = 0.
\ee
We remark that by Kneser's theorem, whenever a solution of the initial value problem is non-unique, then there are a continuous infinity of them (see  \S II.4 of \cite{hartman2002ordinary}). Such solutions are obtained in our case combining the trivial evolution $x(t) \equiv 0$ in an arbitrary interval $t_0 \le t \le t_1$ with any of nontrivial solutions \eqref{exactSolb}, where $t_0$ is replaced by $t_1$. This simple example has deep physical significance within the Richardson picture of particle separation by a turbulent flow~\cite{richardson1926atmospheric}. In fact, equation \eqref{1dmodel} with $x$ denoting the separation between two particles can serve as a toy model for Richardson dispersion  in an ``ideal" K41 velocity field. 
The solution \eqref{exactSolb} implies that particles starting at exactly the same point $x_0=0$ can split and reach finite distances apart {in finite time~\cite{falkovich2001particles};} related phenomena have been addressed in numerous numerical studies, e.g., \cite{sawford2001turbulent,eyink2011stochastic,bitane2013geometry,benveniste2014asymptotic}.

Due to their ubiquitous nature, it is important to develop an understanding of such non-smooth systems, when the classical theory of differential equations fails to provide the unique solutions.  A particularly important question is whether or not there is a natural way to continue solutions past a blowup uniquely (or at least, to provide a non-trivial restriction).
An example where this is accomplished for ordinary differential equations is the theory of $N$-body problem. There, binary (and in certain special cases three and four body) collisions can be regularized through singularity unfolding on the so-called collision manifold~\cite{diacu1992singularities,martinez2000degree,bakker2013understanding}, such that nearby solutions have continuity with respect to initial conditions.  
In partial differential equations, generalized descriptions (such as weak or distributional solutions) become necessary after blowup occurs. Weak solutions are generally highly non-unique and one must again consider the question of selection.
An example where this approach saw great success is in the theory of one-dimensional conservation laws, where entropy conditions can be used to select a \emph{unique} {continuation as an entropy--producing weak (shock) solution,} thereby establishing a predictive theory in that setting \cite{dafermos2010hyperbolic}. Moreover, these entropy conditions are imposed by the vanishing viscosity limit~\cite{lax1973hyperbolic}, which arises as a physically natural regularization of the inviscid problem. In the presence of shock waves, singularities in the characteristic equation lead to non-unique solutions considered backward in time, which is related to the dissipative anomaly of discontinuous solutions~\cite{eyink2015spontaneous}.
 The situation is more complex for higher dimensional conservation laws, where analogous strong results are not generally known.  In three-dimensional incompressible Euler equations, one natural analogue of entropy conditions -- that the solutions dissipate energy -- is now actually known to fail as a selection principle for weak solutions~\cite{de2013dissipative,buckmaster2017onsager,isett2017nonuniqueness}. The question remains as to whether any physical criteria, such as the vanishing viscosity limit, may serve to select unique solutions. Some constructive suggestions in this direction may be inferred from simplified turbulence models~\cite{mailybaev2016spontaneous,mailybaev2017toward,biferale2018rayleigh}. 

Motivated by these ideas, the goal of the present paper is to investigate fundamental constraints on possible ``selected'' solutions, when the non-uniqueness is caused by an isolated non-Lipschitz singularity.  
Our study exploits a class of ordinary differential equations described in Section~\ref{model}, which involve vector fields which are smooth away from the origin. At the origin, the equations are merely H\"{o}lder continuous (providing a cartoon of fields arising in fluid dynamics problems) or even unbounded (mimicking collisional problems for particles with Newtonian potentials). In generic situations, solutions enter the origin in finite time, which we call `blowup'. Since infinite number of solutions also start at the origin, continuation past the blowup is strongly non-unique. 

In these models we uncover some universal characteristics of non-uniqueness in the context of a ``physically relevant'' choice. To accomplish this, we introduce a set of solution-dependent renormalized phase-time variables. In this new coordinate system, approaching and leaving the singularity takes infinite ``time", thereby representing the solution as two different infinitely long evolutions of a renormalized   autonomous system. Section~\ref{sec3} shows how the first evolution, which is governed by one attractor of the renormalized system, determines the universal asymptotic behavior before the blowup. 
To study continuation after blowup, when the uniqueness is typically lost, we introduce in Section~\ref{sec4_post} a regularization scheme, which smooths the vector field inside a small $\nu$--sphere surrounding the origin. Then we consider the limit $\nu \to 0$, called the inviscid limit by the analogy with the smoothing effect of viscosity in fluid equations. {This regularization defines a connection between the two infinitely long renormalized evolutions before and after blowup.} For a wide class  of regularizations, we show that the limiting solutions must leave the origin immediately upon arrival along a trajectory associated with a different attractor of the renormalized system. For example, in the case of a fixed-point attractor studied in Section~\ref{sec5_fp}, this yields a globally unique solution to the problem. A remarkable property of this uniquely selected solution is that it is not sensitive to a particular choice of regularization, i.e. the ``physically relevant'' choice is foreseen by a singular differential equation. In a different situation, when the attractor is a periodic limit cycle studied in Section~\ref{sec6_lc}, such solutions are still non-unique. However, the regularization imposes a highly non-trivial constraint on possible continuations that drastically reduces possible future histories {irrespectively of} the choice of viscous regularization. 

As a consequence, we establish a connection between the theory of attractors in autonomous dynamical systems with the problem of non-uniqueness and continuation {at the non-Lipschitz singularity. It} shows that in a generic case the equation may impose a ``destiny" on a solution even after the uniqueness is lost. Our theorems are coupled with concrete examples illustrating the non-uniqueness and specific choices of solutions after blowup. We end with the Discussion section, where we mention similar phenomena observed recently in turbulence models.
 
\section{{Equations with an isolated non-Lipschitz singularity}}
\label{model}

We study the Cauchy problem for an autonomous system of ordinary differential equations 
\begin{align}\label{ODE}
\dot{\bx} &= {\bold f}(\bx), \quad \quad \bx(t_0)= \bx_0,
\end{align}
{where $\bx \in \mathbb{R}^d$ and the derivative is taken with respect to time $t \in \mathbb{R}$.}
We assume that ${\bold f}$ has an isolated singular (non-Lipschitz) point, which we assign to the origin,
$\bx=\bzero$. We will consider a self-similar form of the singularity, when
\begin{align}\label{ODEb}
\bff(\bx)=r^\alpha \bFF(\by), \qquad  r=|\bx|, \quad  \by= \bx/r
\end{align}
with $\alpha < 1$ and a smooth (continuously differentiable) function $\bFF(\by)$ on a unit sphere $\by \in S^{d-1}$. 
{The function ${\bold f}(\bx)$ is smooth everywhere except at the origin, where it is discontinuous for negative $\alpha$ and H\"{o}lder continuous, ${\bold f}\in C^\alpha(\mathbb{R}^d)$ for $\alpha\in (0,1)$. 
We will separate the radial and spherical components of this function as
\be\label{Fdecomp}
\bFF(\by) = F_r(\by) \by+\bFF_s(\by)
\ee
with $F_r(\by) = \bFF(\by)\cdot \by$ and $\bFF_s$ defined implicitly by \eqref{Fdecomp}. The models \eqref{ODEb} can be thought of as multi-$d$ generalizations of the toy model \eqref{1dmodel} for $\alpha = 1/3$; one can also think of collisions in the N-body problem for $\alpha = -2$, etc. }

Notice that the exponent $\alpha$ {in (\ref{ODE})--(\ref{ODEb})} can be reduced to zero by changing the variables as $\bx_{\textrm{new}} = r^{-\alpha}\bx$, which yields
\begin{align}\label{ODEex}
\dot{\bx}_{\textrm{new}} &= {\bold F}_{\textrm{new}}(\by), \qquad
\bFF_{\textrm{new}}(\by) = (1-\alpha) F_r(\by)\by+\bFF_s(\by), 
\quad r_{\textrm{new}} = |{\bx}_{\textrm{new}}|,
\quad \by = {\bx}_{\textrm{new}}/r_{\textrm{new}}.
\end{align}
The inverse transformation takes the form $\bx = r_{\textrm{new}}^{\alpha/(1-\alpha)}\bx_{\textrm{new}}$, i.e., the origin and infinity in the original variables $\bx$ are mapped, respectively, to the origin and infinity in $\bx_{\textrm{new}}$ for $\alpha < 1$. {Note that the transformation $\bx_{\textrm{new}} = r^{-\alpha}\bx$ is not smooth at the origin, for which reason we will keep the expression (\ref{ODEb}) with a general value of $\alpha < 1$.}

{For $\alpha \in (0,1)$, the function $\bff(\bx)$ is continuous and bounded by $|\bff(\bx)| < r^\alpha \max \bFF(\by)$, where the maximum is taken on the unit sphere $\by \in S^{d-1}$. Hence, the Cauchy problem (\ref{ODE}) possesses a solution globally in time for any initial condition~\cite{hartman2002ordinary}. 
This property extends to any $\alpha < 1$ due to the equivalence imposed by (\ref{ODEex}).} However, the solutions that pass through the singularity at the origin, are not necessarily unique. 

Let us denote by $t_b$ the maximal time such that $r(t) \ne 0$ for $t \in [t_0,t_b)$. This provides the largest time interval, where the uniqueness of the solution is guaranteed due to the Lipschitz continuity. In particular, we have the global uniqueness in the case of $t_b = \infty$. When $t_b<\infty$, the solution reaches the singularity at the origin in finite time, $\lim_{t\to t_b}r(t) =  0$, and we call this scenario the \textit{finite-time blowup}. This definition can be related to collisions or, alternatively, can be motivated by the blowup concept for partial differential equations, where it is linked to the breakdown of Lipschitz continuity. Recall, for example, that solutions of the inviscid Burgers equation have the local form $u(x) \propto -x^{1/3}$ at the blowup time~\cite{pomeau2008wave}. In many physical systems~\cite{eggers2009role,mailybaev2012}, the blowup has a self-similar asymptotic form justifying our choice of the self-similar ``toy-model'' vector field (\ref{ODEb}). 

So long as $r(t) > 0$, we can define the logarithmic radial coordinate as $z(t) = \ln r(t)$. The original solution is then expressed as
\begin{align} \label{xsol}
\bx(t) = e^{{z(t)}} {\by(t)}.
\end{align}
Equations for $z(t)$ and $\by(t)$ are obtained directly from (\ref{ODE}) and read:
\begin{align}\label{eq1}
\dot{\by}&= {e^{-(1-\alpha)z} \bFF_s(\by),}
\quad \by(t_0) = \by_0,
\\ \label{eq2}
\dot{z} &=e^{-(1-\alpha) z} F_r(\by),  \quad\ z(t_0) = z_0.
\end{align}
where the initial conditions are given by $\by_0 = \bx_0/|\bx_0|$ and $z_0 = \ln|\bx_0|$.


\section{Pre-Blowup Dynamics} 
\label{sec3}

For $t < t_b$, we introduce the a new temporal variable $s = s(t)$ defined by
\begin{align}\label{eq_s_of_t}
s(t) = \int_{t_0}^t e^{-(1-\alpha)z(t')} dt'.
\end{align}
The map $s(t)$ is monotonically increasing with $s(t_0) = 0$, and let us denote $s_b = \lim_{t \to t_b} s(t)$. Hence, the inverse monotonically increasing function $t = \tilt(s)$ exists with $\tilde{t}(0) = t_0$ and $t_b = \lim_{s \to s_b} \tilde{t}(s)$. We denote $\tby(s) :=\by(\tilt(s))$ and $\tz(s) :=z(\tilt(s))$; here and below we use tildes to denote  functions of the new time variable $s$. By the inverse function theorem, the derivative of $\tilt(s)$, obtained from (\ref{eq_s_of_t}), can be written in the form $d\tilt/ds = e^{(1-\alpha)\tz(s)}$ and  integrated as
\begin{align}\label{eq_t_of_s}
\tilt(s) = t_0+\int_{0}^s e^{(1-\alpha)\tz(s')} ds'.
\end{align}

Using $s(t)$ as the new temporal variable in system (\ref{eq1})--(\ref{eq2}), a simple computation yields the renormalized system in the form 
\begin{align} \label{yEqn}
d\tby/ds &= \bFF_s( {\tby}), \quad \quad \tby(0)=\by_0, \\  \label{zEqn}
d\tz/ds&= F_r(\tby),  \quad \quad \ \tz(0)=z_0.
\end{align}
In this system, the first equation (\ref{yEqn}) is uncoupled, and the second equation (\ref{zEqn}) is integrated as
\begin{align} 
\label{zEqnInt}
\tz(s) = z_0+\int_0^{s} F_r(\tby(s'))ds'.
\end{align}
Since the functions $\bFF_s$ and $F_r$ are continuous on the unit sphere $S^{d-1}$ and therefore bounded, solutions $\tby(s)$ and $\tz(s)$ exist globally in renormalized time and are unique.  

Let us show that $s_b = \infty$. In the case of blowup, $t_b < \infty$, we have $z(t) = \ln r(t) \to -\infty$ as $t \to t_b$. Since the function $F_r$ is bounded, the corresponding behavior $\tz(s) \to -\infty$ as $s \to s_b$ in (\ref{zEqnInt}) yields $s_b = \infty$. Therefore, using $\tilde{t}(s)$ from (\ref{eq_t_of_s}), one has
\be\label{timeblowup}
t_b =  \lim_{s\to \infty}\tilt(s)  = t_0+\int_{0}^\infty e^{(1-\alpha)\tz(s')} ds'.
\ee 
The same conclusion can be drawn in the case of no blowup, $t_b = \infty$, because $\tilde{t}(s)$ is finite for any finite value of $s$.

The solution $\bx(t)$ of the original system (\ref{ODE}) for $t \in [t_0,t_b)$ can be recovered from the global histories $\tby(s)$ and $\tz(s)$ with $s= {s}(t)$  given by (\ref{eq_s_of_t}) through the transformations
\begin{align}\label{xtrans}
\bx(t) = \tilde{\bx}({s} (t)), \qquad \tilde{\bx}(s)= e^{{\tilde{z}(s)}} {\tilde{\by}(s)}.
\end{align}
We see that the key property of the renormalization is to extend the blowup time (if it exists) to infinity, while maintaining the equations in autonomous form (\ref{yEqn})--(\ref{zEqn}).

Define the following upper and lower renormalized-time averages:
\begin{align}\label{eq_prop_1}
  \underline{ F_r } =  \liminf_{s \to \infty}   \frac{1}{s}\int_0^s F_r(\tby(s'))ds',
  \qquad 
 \overline{ F_r } =  \limsup_{s \to \infty}   \frac{1}{s}\int_0^s F_r(\tby(s'))ds'.
\end{align}
Recall that these averages are finite because $F_r$ is bounded. 

\begin{prop}\label{prop:infRenormTime}
The following conditions on the averages \eqref{eq_prop_1} characterize blowup:\\
(a) If $\underline{ F_r }  > 0$, then $t_b=\infty$ and $\lim_{t\to\infty}r(t)= \infty$. \\[3pt]
(b)  If  $\overline{ F_r }  < 0$, then the solution $\bx(t)$  blows up at finite time $t_b < \infty$, i.e.,  $\lim_{t\to t_b}r(t)= 0$. 
\end{prop}

\begin{proof} 
By definition, for any $\varepsilon > 0$ there exists an $s_\ve>0$ such that 
\be\label{ineq1}
\underline{ F_r } -\varepsilon/2<  \frac{1}{s}\int_0^s F_r(\tby(s'))ds' 
< \overline{ F_r } +\varepsilon/2 \quad
\textrm{for} \quad
s > s_\varepsilon.
\ee
The value of $s_\ve$ can be chosen sufficiently large such that 
\be\label{ineq2}
|z_0| < \varepsilon s_\varepsilon/2. 
\ee
Using (\ref{ineq1})--(\ref{ineq2}) in expression  \eqref{zEqnInt}, we find
\be\label{zbnds}
\left(\underline{ F_r } -\varepsilon\right)s < \tilde{z}(s) 
< \left(\overline{ F_r } +\varepsilon\right)s
\quad
\textrm{for} \quad
s > s_\varepsilon.
\ee

\noindent (a)  If $\underline{ F_r }  > 0$, then the first inequality in (\ref{zbnds}) with $\ve = \underline{ F_r }/2 $ yields 
\be\label{ineq3}
\tilde{z}(s) > \underline{ F_r }\,s/2 
\quad
\textrm{for} \quad
s > s_\varepsilon.
\ee
Using this estimate in (\ref{eq_t_of_s}), we obtain the lower bound: 
\begin{align}
\label{eqt_int}
\tilt(s) > \tilt(s_\varepsilon)+\int_{s_\varepsilon}^s e^{(1-\alpha)\underline{ F_r }\,s'/2 }ds'
\quad
\textrm{for} \quad
s > s_\varepsilon.
\end{align}
Since $\alpha < 1$ and $\underline{ F_r }  > 0$, the integral in (\ref{eqt_int}) diverges as $s \to \infty$ and  we have $t_b=\lim_{s\to\infty} \tilt(s)=\infty$. This proves that the solution $\bx(t)$ does not blow up in finite time.  From Eq.~\eqref{ineq3}, it follows that $\lim_{t\to\infty} r(t)=\lim_{s\to\infty} e^{\tilde{z}(s)} =\infty$.
\vspace{2mm}

\noindent (b) If $\overline{ F_r } < 0$, a similar argument with $\varepsilon = -\overline{ F_r }/2> 0$ and the second inequality in  (\ref{zbnds})  yields 
\be\label{ineq5}
\tilde{z}(s) < \overline{ F_r }\,s/2
\quad
\textrm{for} \quad
s > s_\varepsilon. 
\ee
This guarantees that $\lim_{s\to \infty}\tilde{z}(s) = -\infty$ and the integral in (\ref{eq_t_of_s}) converges to a finite value $t_b = \lim_{s\to \infty}\tilt(s) < \infty$. Thus, the solution reaches the origin $\lim_{t\to t_b} r(t) =0$ in a finite time $t_b$.  
\end{proof}

By Proposition \ref{prop:infRenormTime}, the \emph{blowup is determined by the solutions of the renormalized system} as $s\to\infty$. Thus, it is natural to associate the blowup with dynamical attractors. 
An attractor of the autonomous dynamical system (\ref{yEqn}) is a nonempty, compact, invariant set $\mathcal{A}$ that has a neighborhood $U_0$ with the property $\cap_{s\ge 0}U_s = \cA$, where $U_s$ is the neighborhood $U_0$ transported by the system flux at time $s > 0$, see, e.g.,~\cite{eckmann1985ergodic,hurley1982attractors}. The basin of attraction, $\mathcal{B}(\cA)$, is the set of all points that approach $\mathcal{A}$ asymptotically under $s \to \infty$. According to Proposition \ref{prop:infRenormTime}, if $\by_0 \in \mathcal{B}(\mathcal{A})$, then the blowup dynamics is controlled by the average of the function $F_r$ on the attractor. 

\begin{defin} 
\label{def1}
We say that the attractor of the system \eqref{yEqn} is \textit{focusing} if $\overline{ F_r }  < 0$ for any $\by_0 \in \mathcal{B}(\mathcal{A})$. Similarly, we call the attractor \textit{defocusing} if $\underline{ F_r }  > 0$ for any $\by_0 \in \mathcal{B}(\mathcal{A})$. 
\end{defin}

For simple attractors like an asymptotically stable fixed point or a limit cycle, there exists an average value, $\langle F_r \rangle = \underline{ F_r } = \overline{ F_r }$, which is independent of the initial condition $\by_0 \in \mathcal{B}(\mathcal{A})$. For example, if $\cA$ is a fixed point $\{\by_*\}$, then $\langle F_r\rangle = F_r(\by_*)$.  If $\cA$  is a limit cycle, then $\langle F_r\rangle$  is the average of $F_r(\tby(s))$ over one period of the attractor.  
The average $\langle F_r \rangle$ can also exist in more complex situations, in the case of  quasi-periodic and chaotic attractors under proper ergodicity assumptions.   
For example for sufficiently mixing flows, $\langle F_r\rangle$  exists and is computed as an average of $F_r(\tby(s))$ over the attractor, with respect to the SRB measure; {see~\cite{bowen1975ergodic,eckmann1985ergodic} for precise definitions.}


With the well-defined average value, the property of the attractor to be focusing ($\langle F_r \rangle < 0$) or defocusing ($\langle F_r \rangle > 0$) becomes generic. The only exception is given by the degeneracy condition $\langle F_r \rangle = 0$.
The immediate consequence of Proposition \ref{prop:infRenormTime} is:
 
\begin{thm}\label{th:EnterAttractor}
If $\by_0 \in \mathcal{B}(\mathcal{A})$ for a defocusing attractor $\cA$ of system  (\ref{yEqn}), then $t_b=\infty$ and $\lim_{t\to\infty}r(t)= \infty$. 
If $\by_0 \in \mathcal{B}(\mathcal{A})$ for a focusing attractor $\cA$, then the solution $\bx(t)$ blows up in finite time and the attractor describes the asymptotics of the spherical variables:  $\lim_{t\to t_b}\mathrm{dist}(\by(t),\cA) = 0$.
\end{thm}

 This result provides a natural tool for characterizing and classifying possible types of blowup. In general, when $\langle F_r\rangle < 0$, the integrals  (\ref{eq_t_of_s}) and  (\ref{zEqnInt}) yield the estimates
\begin{align} 
\label{zEqnIntB}
\tz(s) \sim \langle F_r\rangle s,\quad
t_b-\tilt(s) \sim \exp\left[(1-\alpha)\langle F_r\rangle s\right]
\quad \textrm{for} \quad
s \to \infty,
\end{align}
which lead to the power-law asymptotic for $r = e^z$ as
\begin{align} 
\label{zEqnIntD}
r(t) \sim (t_b-t)^{\frac{1}{1-\alpha}}
\quad \textrm{for} \quad
t \to t_b.
\end{align}

We remark that the discussed scenarios of blowup in the low-dimensional singular system (\ref{ODE})--(\ref{ODEb}) closely resemble the blowup dynamics in partial differential equations~\cite{pomeau2008wave,eggers2009role} and infinite-dimensional turbulence models~\cite{dombre1998intermittency,mailybaev2012,mailybaev2012c}, where the asymptotic blowup dynamics is associated to the attractor of a renormalized system. This justifies the interpretation of (\ref{ODE})--(\ref{ODEb}) as a toy-model for studying evolution before and after the blowup.

To conclude this section, we give two  concrete examples.

\begin{example}\label{Ex0}\normalfont
In the simplest case an attractor $\cA$ is a fixed point, say $\cA =\{\by_*\}$ with $\bFF_s(\by_*) = \mathbf{0}$. Assuming that $\by_0\in \mathcal{B}(\{\by_*\})$, we find the average value $\langle F_r\rangle = F_r(\by_*)$.  A negative value of $F_r(\by_*)$ guarantees the finite-time blowup. For $\bx_0 = r_0\by_*$ and $r_0 = e^{z_0}$, the fixed-point solution $\tby(s) \equiv \by_*$ with expressions (\ref{zEqnInt}) and (\ref{eq_t_of_s}) yield 
\begin{align} \label{eqSol1}
\tz(s) = z_0+F_r(\by_*)s,\quad
\tilt(s) = t_b-\frac{\exp\left[(1-\alpha)(z_0+F_r(\by_*)s)\right]}{(\alpha-1)F_r(\by_*)},
\end{align}
where the blowup time \eqref{timeblowup} is given by
\begin{align} \label{eqFPgen2}
t_b 
= {t_0+\frac{\exp[(1-\alpha)z_0]}{(\alpha-1)F_r(\by_*)}}
= t_0+\frac{r_0^{1-\alpha}}{(\alpha-1)F_r(\by_*)}.
\end{align}
Recall that  $(\alpha-1)F_r(\by_*) > 0$ for negative $F_r(\by_*)$ and $\alpha < 1$. The second expression in \eqref{eqSol1} can be used to solve the equation $t = \tilt(s)$ with respect to $s = s(t)$. 
Then the solution $\bx(t) = r(t)\by(t)$ with $\by(t) \equiv \by_*$ and $r(t) = e^{\tz(s(t))}$ is obtained in the form
\begin{align} \label{eqFPgen}
\bx(t)= r(t)\by_*,\quad
r(t) = \left[(\alpha-1)F_r(\by_*)(t_b-t)\right]^{\frac{1}{1-\alpha}}.
\end{align}
In the more general case, when $\by_0$ belongs to the basin of attraction of $\by_*$, expression (\ref{eqFPgen}) provides the asymptotic form of the solution before the blowup, as $t \to t_b$.
\end{example}

\begin{example}\label{Ex1}\normalfont
Consider a system of equations (\ref{ODE})--(\ref{ODEb}) for any $\alpha < 1$ and $\bFF(\by) = (F_1,F_2)$ with
	\begin{equation}
	F_1(\by) = y_1^2+y_1y_2+y_1y_2^2,\quad
	F_2(\by) = y_1y_2+y_2^2-y_1^2y_2.
	\label{eq2.1}
	\end{equation}
Phase diagram of this system is shown in Fig.~\ref{fig1}(a). For all initial conditions on the left half-plane $x_1 < 0$ and semi-axis $x_1 = 0$, $x_2 < 0$, solutions enter the singularity at the origin in finite time (blowup). On the other hand, all solutions in the upper half-place $x_2 > 0$ and semi-axis $x_1 > 0$, $x_2 = 0$ originate at the origin. The curves in the fourth quadrant $x_1 > 0$, $x_2 < 0$ never hit the singularity.

\begin{figure}[t]
\centering
\includegraphics[width=0.75\textwidth]{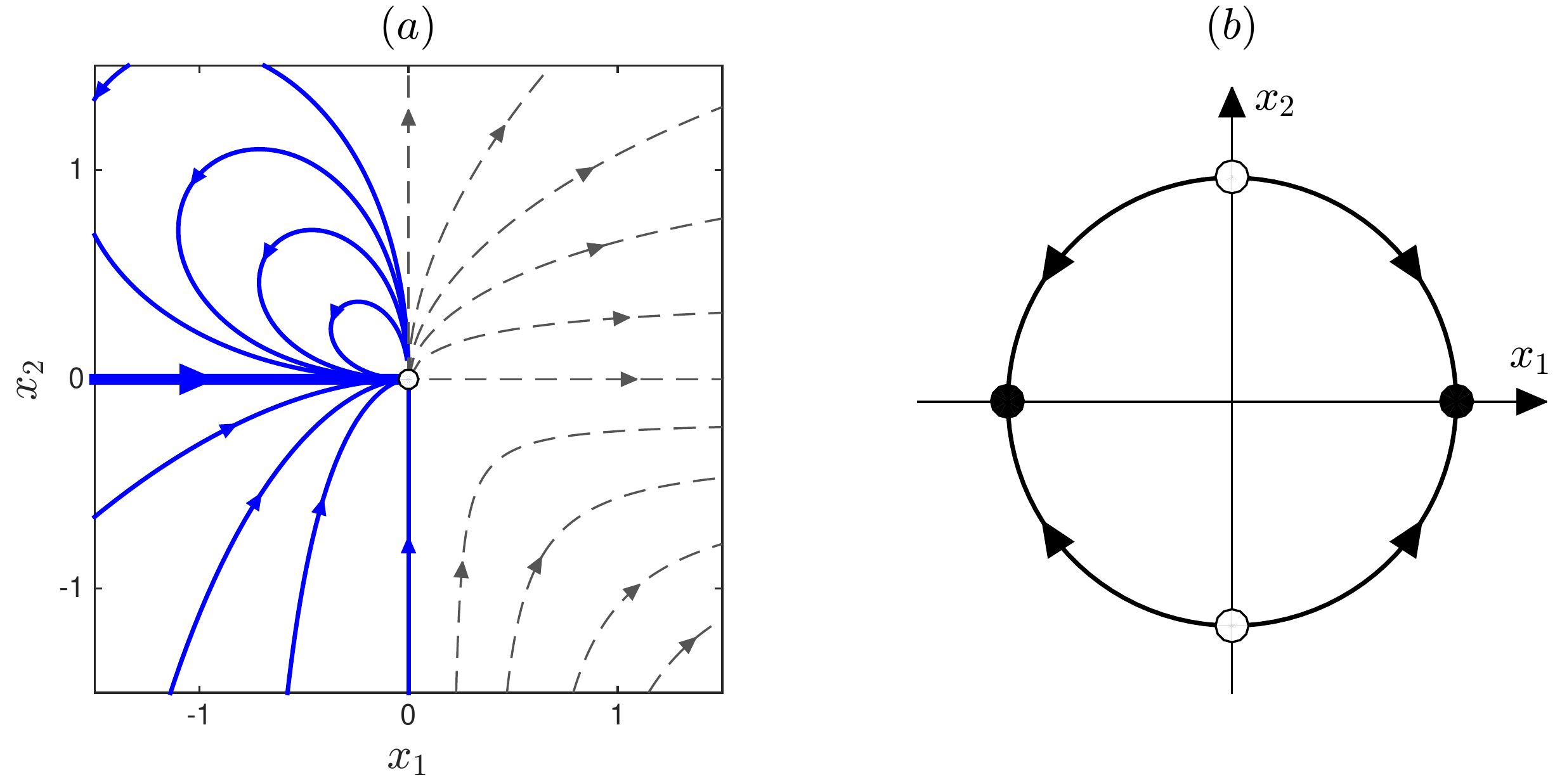}
\hspace{5mm}
\caption{(a) Vector field of Eqs.~(\ref{ODE})--(\ref{ODEb}) and (\ref{eq2.1}) with representative solutions. Curves entering the origin are colored \textcolor{blue}{blue}; points along them correspond to initial conditions which lead to `finite-time blowup'. 
(b) Phase portrait of the renormalized system on a circle. Attractors are shown by \textbf{black} circles and repellers by white circles.}
\label{fig1}
\end{figure}

We can express  $\by = (\cos\varphi,\sin\varphi)$, where the real variable $\varphi \in S^1$ describes the angular dynamics. Then the radial and circular components {of the vector field} are introduced as
	\begin{equation}
	{\left(\begin{array}{c} F_1 \\ F_2 \end{array}\right) 
	= \left(\begin{array}{c} y_1 \\ y_2 \end{array}\right) F_r
	+ \left(\begin{array}{c} -y_2 \\ y_1 \end{array}\right) F_\varphi,}\qquad
	F_r(\varphi) 
	= \sqrt{2}\sin \left(\varphi+\frac{\pi}{4}\right),\quad
	F_s(\varphi) 
	= {-\frac{\sin 2\varphi}{2}.}
	\label{eq2.3}
	\end{equation}
For the renormalized system (\ref{yEqn})--(\ref{zEqn}) we obtain 
	\begin{align} \label{phiEqn}
	d{\tvf}/ds = F_s( {\tvf}),  \quad
	d{\tz}/ds = F_r(\tvf),
	\end{align}
for $\tvf(s):= \varphi(\tilde{t}(s))$.
Dynamics of the equation for $\tvf$ in (\ref{phiEqn}) is very simple as it possess only two fixed-point attractors at $\tvf = 0$ and $\pi$, see Fig.~\ref{fig1}(b). The corresponding basins of attraction are separated by the two unstable fixed points at $\tvf = \pm\pi/2$. Since $F_r(0) = 1 > 0$ for the first attractor, it does not correspond to the blowup. The second attractor with $\tvf = \pi$ has $F_r(\pi) = -1 < 0$ with the basin of attraction $\pi/2 < \tvf < 3\pi/2$. Therefore, it describes the blowup from the initial condition at any point of the left half-plane $x_1 < 0$ (basin of attraction). This is confirmed in Fig.~\ref{fig1}(a) demonstrating the phase portrait of original system {\eqref{ODE}--\eqref{ODEb}.} All blue curves enter the origin in finite time.

The stable fixed-point defines the blowup solution (\ref{eqFPgen}), which in our example takes the form
\begin{align} \label{xsolFP}
\bx(t)=-
\begin{pmatrix}\left[(1-\alpha)(t_b-t)\right]^{\frac{1}{1-\alpha} }
\\
0
\end{pmatrix},\quad  0 \le t \le t_b.
\end{align}
By Theorem~\ref{th:EnterAttractor}, the solution (\ref{xsolFP}) is asymptotic for all solutions that end at the singularity from the left half-plane $x_1 < 0$; see the bold blue line in Fig.~\ref{fig1}(a).   These also exists a single solution corresponding to the unstable fixed point $\tilde{\varphi} = 3\pi/2$, see Fig.~\ref{fig1}(b). This solution is similarly found as
\begin{align} \label{xsolFPalt}
\bx(t)=-
\begin{pmatrix}0
\\
\left[(1-\alpha)(t_b-t)\right]^{\frac{1}{1-\alpha} }
\end{pmatrix},\quad  0 \le t \le t_b.
\end{align}
This solution  is non-generic as it corresponds to a zero measure set of initial conditions.
Thus, expression (\ref{xsolFP}) describes asymptotically the only generic blowup scenario in the system.   
\end{example}

\section{Post-Blowup Dynamics } \label{sec4_post}

Our system is not Lipschitz continuous at the origin and, hence, there can be multiple solutions starting at the singularity; see, for example Fig.~\ref{fig1}(a), where all solutions with $x_2 > 0$ originate at the singularity at finite time. This prevents one from uniquely defining the solution globally in time and motivates the study of regularized problems for the search of selection principles.

\subsection{$\nu$--Regularization and ``inviscid'' limit} \label{sec:nuReg}

 In this section,  we consider a class of regularized problems $\dot{\bx} = \bff^\nu(\bx)$ which have unique global solutions and provide a selection rule by taking the limit $\nu \to 0$ in which the regularization is removed. 
A priori, there are infinitely many different ways to regularize the vector field.   For example, one can take the convolution $\mathbf{f}^\nu =  G^\nu *\mathbf{f} $ with any smooth scaled mollifier $G^\nu$ which approximates the identity $G^\nu\to \delta$ as $\nu\to 0$.  
Often, a physical application determines a relevant regularization.
 
For analytical convenience,  we consider the family $\bff^\nu(\bx)$ obtained by smoothing the function $\bff(\bx)$ in Eq.~(\ref{ODEb}) inside a sphere of radius $\nu$ centered at the origin. Then the vector field is constructed by patching together the regularized and the original fields inside and outside the sphere. Due to the self-similar form of $\bff(\bx)$, it is convenient to define the regularization for all $\nu > 0$ with the same function $\bG(\bx)$ scaled properly inside the $\nu$-sphere.  More specifically, 
we say $\bx = \bx^\nu(t)$ is a solution of a \emph{$\nu$--regularized problem}  if it solves
\begin{align}\label{eqn:epsReg}
 \dot{\bx} = \bff^\nu(\bx) ,\quad   \bff^\nu(\bx):= 
 \begin{cases} 
	r^\alpha\bFF(\bx/r) & r > \nu;\\
	\nu^\alpha \bG(\bx/\nu) & r \leq \nu,
 \end{cases}\qquad 
  {\bx}(t_0)=\bx_0,
\end{align}
where the function $\bG(\bx)$ is designed so that $\bff^\nu(\bx)$ is continuously differentiable everywhere.  Such regularization is demonstrated schematically in Fig.~\ref{fig2}(a).
It is easy to see that the regularized solution $\bx^\nu(t)$ exists and is unique globally in time. 

\begin{figure}
\centering
\includegraphics[width=0.35\textwidth]{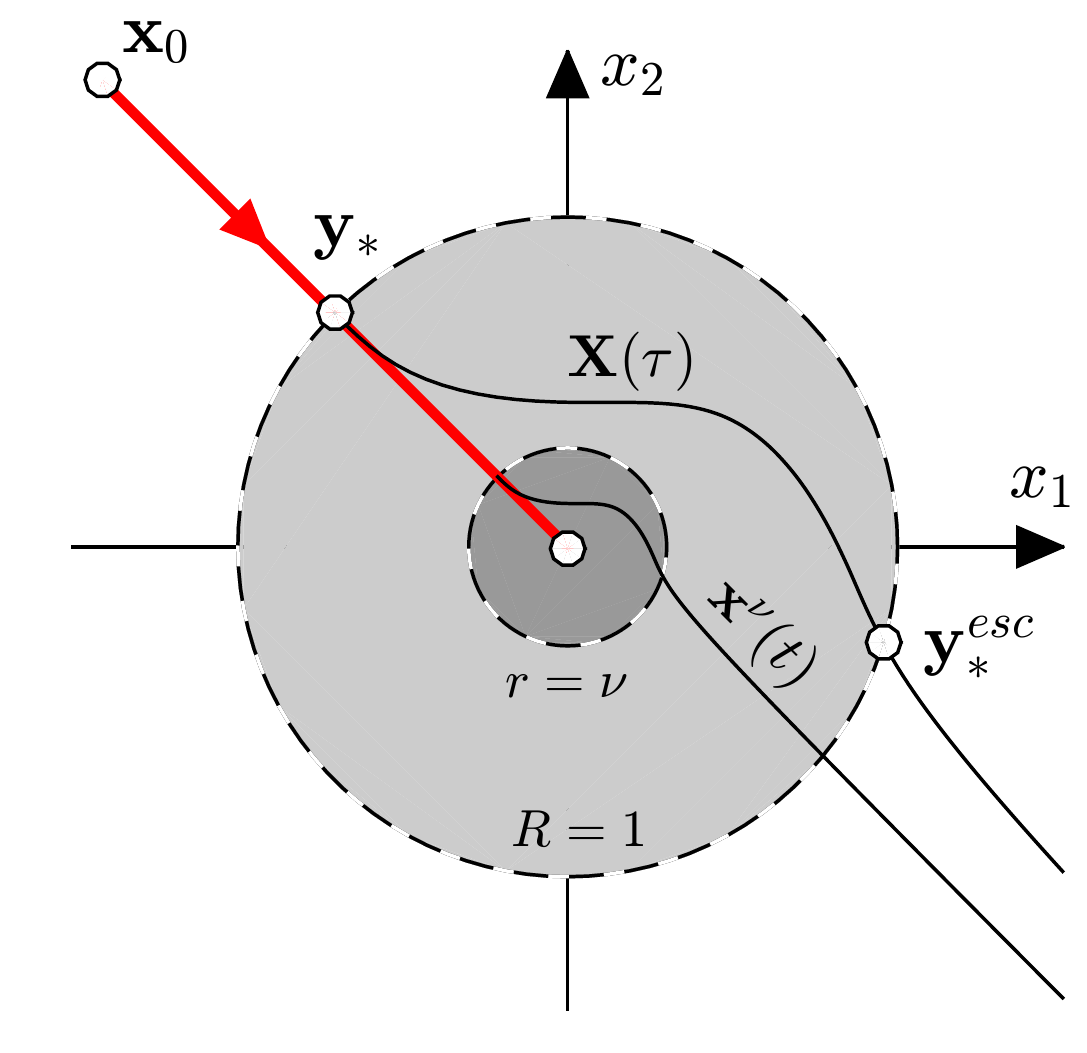}
\caption{
(a) Outline of the regularized system \eqref{eqn:epsReg}:  The \red{red} line corresponds to blowup solution for $\nu = 0$.  The \textbf{black} curve $\bx^{\nu}(t)$ shows the regularized solution with the equation modified in the \textbf{\textcolor{gray}{gray}} disk of small radius $\nu > 0$. The other \textbf{black} curve $\bX(\tau)$ is a solution of the rescaled system \eqref{ODErescale}, which corresponds to $\nu = 1$.
}
\label{fig2}
\end{figure}

There is some similarity between our regularization and the role of viscosity in fluid dynamics: both change the system at small scales, which are responsible to the blowup. Motivated by this analogy, the limit $\nu \to 0$ can be termed the \textit{inviscid limit}.
We now show that $\bx^\nu(t)$ converges (along subsequences) to a solution of the original singular system  \eqref{ODE}.  

\begin{thm}\label{convSub}
Given any initial condition $\bx^{\nu}(t_0) = \bx_0$ and any finite time interval $[t_0,t_1]$, there exists a subsequence $\{\nu_n\}_{n\geq 0}$  of solutions  $\{\bx^{\nu_n}(t)\}_{n\geq 0}$  of the $\nu$--regularized problem  \eqref{eqn:epsReg} with $\nu_n \to 0$ as $n\to\infty$, which converges uniformly to a limit 
\begin{equation}\label{limx}
\bx(t) = \lim_{n\to \infty}\bx^{\nu_n}(t)
\quad
\textrm{for} \quad
t \in [t_0,\,t_1].
\end{equation}
The limiting function $\bx(t)$ is a solution of the original system \eqref{ODE} with $\bx(t_0) = \bx_0$.
\end{thm}

\begin{proof}
Let us first consider the case $\alpha \in [0,\,1)$. Equation \eqref{eqn:epsReg} provides the estimate for the time derivative of the norm $r^{\nu}(t) = |\bx^{\nu}(t)|$ as
\begin{align}\label{eq1b}
\dot{r}^\nu & \le A\max(r^{\nu},\nu)^\alpha \le A(r^{\nu}+\nu)^\alpha,
\end{align}
where $A \ge 0$ is the maximum value of the norms {$|\bFF(\by)|$ for $\by \in S^{d-1}$ and $|\bG(\bx)|$ for $|\bx| \le 1$.} 
The extremal solution $r^{\nu}_{ext}(t)$ of differential inequality \eqref{eq1b} can be found as 
\begin{equation}
\label{eq1ext}
r^{\nu}_{ext}(t) = \left[A(1-\alpha)(t-t_0)+(r_0+\nu)^{1-\alpha}\right]^{1/(1-\alpha)}-\nu,
\end{equation}
where $r_0 = r^{\nu}(t_0) = |\bx_0|$. The function (\ref{eq1ext}) in the interval $t \in [t_0,t_1]$ attains its maximum at the final time $t_1$.
In particular, given the initial value $r_0$, a finite interval $[t_0,t_1]$ and fixing an arbitrary $\nu_0 > 0$, there is a ball of finite radius $r_b = \left[A(1-\alpha)(t_1-t_0)+(r_0+\nu_0)^{1-\alpha}\right]^{1/(1-\alpha)}$ such that
\begin{align}\label{eq1bExtra1}
r^\nu(t) & \le r_b 
\textrm{ \ for any \ }t\in[t_0,t_1],\ \nu \in (0,\,\nu_0]. 
\end{align}
In particular, solutions of \eqref{eqn:epsReg} are \emph{bounded uniformly in $\nu$}. 

We will now show that the family $\{\bx^\nu(t)\}_{\nu>0}$ is equicontinuous.   To see this, note that solutions of \eqref{eqn:epsReg} satisfy the integral form of the equation:
\begin{equation}\label{regIdent}
\bx^\nu(t)= \bx_0+\int_{t_0}^t \bff^\nu(\bx^\nu(t'))dt',\quad \nu > 0.
\end{equation}
Let $B$ be the maximum value of the norm $|\bff^{\nu}(\bx)|$ in the compact set $\{\bx:|\bx | \le r_b\}$ and $\nu \in [0,\,\nu_0]$. Then the relations (\ref{eq1bExtra1}) and (\ref{regIdent}) guarantee that
\begin{equation}\label{liptime}
|{\bx}^\nu(t)-{\bx}^\nu(t')|\leq  B |t-t'|
\textrm{ \ for any \ }t,t'\in[t_0,t_1],\ \nu \in (0,\,\nu_0]. 
\end{equation}
This inequality proves that the family of solutions $\{\bx^\nu(t)\}_{\nu>0}$ is \emph{equicontinuous}. Thus, by the Arzel\`{a}--Ascoli theorem (see, e.g. \cite{rudin1964principles}), there exists a subsequence $\nu_0,\nu_1,\nu_2,\ldots$ such that $\lim_{n\to \infty}\nu_n = 0$ and the corresponding solutions converge uniformly to the continuous function (\ref{limx}). Finally, since the function $\bff^\nu(\bx)$ is uniformly continuous in the compact set $|\bx| \le r_b$, $\nu \in [0,\,\nu_0]$, one can pass to the limit in Eq. \eqref{regIdent}, which yields 
  \begin{equation}
\bx(t)= \bx_0+\int_{t_0}^t \bff(\bx(t'))dt'.
\end{equation}
This implies that the limiting solution $\bx(t)$ solves the original equation \eqref{ODE}.

In the case of negative $\alpha$ we can use the reduction similar to (\ref{ODEex}). One can show that the transformation 
{\begin{equation}\label{eq_corrA1}
\bx_{\textrm{new}} = r^{-\alpha}\bx,
\qquad
\nu_{\textrm{new}} = \nu^{1-\alpha}
\end{equation}}
reduces system \eqref{eqn:epsReg} to the form 
\begin{align}\label{eqn:epsReg_nrm2}
 \dot{\bx}_{\textrm{new}} =  
 \begin{cases} 
	\bFF_{\textrm{new}}(\by) & r_{\textrm{new}} > \nu_{\textrm{new}};\\
	\bG_{\textrm{new}}(\bx_{\textrm{new}}/\nu_{\textrm{new}}) & r_{\textrm{new}} \leq \nu_{\textrm{new}},
 \end{cases}\qquad 
  \bx_{\textrm{new}}(t_0)=|\bx_0|^{1-\alpha}\by_0,
\end{align}
where $\bFF_{\textrm{new}}(\by)$ is given in \eqref{ODEex} and 
\begin{equation}\label{eqn:epsReg_nrm3}
\bG_{\textrm{new}}(\bx_{\textrm{new}}) = r_{\textrm{new}}^{-\alpha/(1-\alpha)}
\left[(1-\alpha) G_r\left(r_{\textrm{new}}^{1/(1-\alpha)}\by\right)\by
+\bG_s\left(r_{\textrm{new}}^{1/(1-\alpha)}\by\right)\right],
\end{equation}
with $G_r(\bx) = \by \cdot \bG(\bx)$ and $\bG_s(\bx) = \bG(\bx)-G_r(\bx)\by$ being the radial and spherical parts of $\bG(\bx)$. The function $\bG_{\textrm{new}}$ in \eqref{eqn:epsReg_nrm3} is continuous in the case of $\alpha < 0$. It remains to notice that the system (\ref{eqn:epsReg_nrm2}) corresponds to the case $\alpha = 0$, for which the statement of the theorem was already proved; {note that this proof required only the continuity of $\bG$}.
\end{proof}

\subsection{Selection by renormalization}
\label{sec4.2}

Although convergent subsequences are guaranteed by Theorem~\ref{convSub}, the limits are generally not unique after the blowup time $t_b$. Different sequences $\nu_n \to 0$ may lead to different {limits, as we show with explicit examples below.} Despite this non-uniqueness, we will see in this section that the regularization procedure drastically decreases the number of ``choices'' that the system can make after the blowup. Most importantly, we will show a counterintuitive fact that this set of possible choices is controlled primarily by the properties of the original singular system, rather than by a specific form of the regularization.

In this paper we only address the post-blowup dynamics in the simplest case, when the blowup is self-similar, i.e., it corresponds to the fixed-point attractor $\by_*$ of the renormalized system (\ref{yEqn}) with $F_r(\by_*) < 0$; see Theorem~\ref{th:EnterAttractor} {and Example~\ref{Ex0}.} Thus, the angular part of our initial condition is assumed to be in the basin of attraction $\mathcal{B}(\{\by_*\})$.

We start by focusing on a specific initial condition $\bx_0 = r_0\by_*$ for some $r_0 > 0$. As we showed in the previous section, this initial condition leads to solution (\ref{eqFPgen}), which blows up at finite time $t_b$ given in (\ref{eqFPgen2}). The same solution is valid for the regularized system (\ref{eqn:epsReg}) in the interval $t_0 \le t \le t_{ent}^\nu$, where
\begin{align}\label{eqRN2}
t_{ent}^\nu := t_b-\frac{\nu^{1-\alpha}}{F_r(\by_*)(\alpha-1)}
\end{align}
is the time when solution (\ref{eqFPgen}) hits the sphere $\{\bx :r = \nu\}$ and starts to be affected by the regularization. 
First,  we show that the limit $\nu \to 0$ reduces to the scaling limit for a single function $\bX(\tau)$ defined by
\begin{align}\label{eqRN1}
\bx^\nu(t)=\nu\bX(\tau^\nu(t)), \quad \tau^\nu(t) := \nu^{-(1-\alpha)} (t-t_b).
\end{align}
By matching the representation (\ref{eqRN1}) with (\ref{eqFPgen}), we obtain
\begin{align}\label{ODErescaleIC}
\bX(\tau) = \left[F_r(\by_*)(1-\alpha)\tau\right]^{\frac{1}{1-\alpha}} \by_*,
\quad 
\tau_0 \le \tau \le \tau_{ent},
\end{align}
where $\tau_0$ and $\tau_{ent}$ correspond to $t_0$ and $t_{ent}^\nu$ via
\begin{align}\label{eqRN3}
\tau_0 := -\nu^{-(1-\alpha)} (t_b-t_0),\quad
\tau_{ent} :=-\nu^{-(1-\alpha)} (t_b-t_{ent}^\nu)=  -\frac{1}{F_r(\by_*)(\alpha-1)}.
\end{align}
Past the time $\tau> \tau_{ent}$, we find the behavior of  $\bX(\tau)$ by substituting (\ref{eqRN1}) into (\ref{eqn:epsReg}) to obtain:
\begin{align}\label{ODErescale}
\frac{d\bX}{d\tau}&= 
 \begin{cases}
R^\alpha \bFF(\bY),& R>1\\
\phantom{R^\alpha} \bG( \bX), & R\leq 1
 \end{cases}, \qquad \bX(\tau_{ent}) = \by_*,
\end{align}
where $R = |\bX|$ and $\bY = \bX/R$.  Thus both the equations \eqref{ODErescale} and initial conditions determining $\bX(\tau)$ are completely independent of $\nu$, and $\bx^\nu$ may be genuinely obtained by the scaling \eqref{eqRN1}. Equation (\ref{ODErescale}) has a unique global solution $\bX(\tau)$. Due to the scaling property (\ref{eqRN1}), the inviscid limit $\nu \to 0$ depends on the behavior of the solution $\bX(\tau)$ at large $\tau$, see Fig.~\ref{fig2}. We will consider the following two generic scenarios: 

\begin{defin}
\label{def2}
Given the solution $\bX(\tau)$ of (\ref{ODErescale}), we say that the regularization is \textit{expelling}, if there exists a finite time $\tau_{esc}>\tau_{ent}$ such that $R(\tau_{esc}) = 1$ and the solution is outside the regularization region, $R(\tau) > 1$, for all $\tau > \tau_{esc}$. We say that the regularization is \textit{trapping}, if the solution stays (or returns to) the regularization region, $R(\tau) \le 1$, for arbitrarily large $\tau$ and, additionally, remains bounded, $R(\tau) \le R_b < \infty$ for all $\tau > \tau_{ent}$. 
\end{defin}

When the regularization is trapping, one can typically expect that system (\ref{ODErescale}) has an attractor, which lies in the bounded region $R(\tau) \le R_b$ and at least partially belongs to the regularization region $R \le 1$. In this case, it is sufficient that the initial point $\by_*$ belongs to the basin of attraction.

In the case of expelling regularization, equation (\ref{ODErescale}) reduces to $\dot{\bX} = R^\alpha \bFF(\bY)$ for $\tau \ge \tau_{esc}$. Thus, we can follow the same steps as {in (\ref{xsol}), (\ref{eq_s_of_t})} and define 
\begin{align}\label{renormCoord2}
Z(\tau) = \ln R(\tau),\quad  
S(\tau) = \int_{\tau_{esc}}^{\tau} e^{-(1-\alpha)Z(\tau')} d\tau'.
\end{align}
We denote by $\tau = \ttau(S)$ the inverse of the monotonously increasing function $S = S(\tau)$ and define $\tbY(S):=  \bY(\ttau(S))$ and  $\tZ(S):=Z(\ttau(S))$. Similarly to (\ref{eq_t_of_s}), we have
\begin{align}\label{eq_T_of_S}
\ttau(S) = \tau_{esc}+\int_{0}^S e^{(1-\alpha)\tZ(S')} dS'.
\end{align}
In analogy to (\ref{yEqn})--(\ref{zEqn}), the renormalized equations for $\tbY(S)$ and  $\tZ(S)$ are written as
\begin{align} \label{yEqn2}
d\tbY/dS&= \bFF_s( {\tbY}), \quad \quad \tbY(0) = \by_*^{esc}, \\  \label{zEqn2}
d\tZ/dS&= F_r(\tbY),  \quad \quad \ \tZ(0) = 0,
\end{align}
where the initial conditions follow from the assumption $\bX(\tau_{esc}) = \by_*^{esc}$ with $Z(\tau_{esc}) = \ln |\bX(\tau_{esc})| = 0$ and $S(\tau_{esc}) = 0$.
The solution of (\ref{yEqn2})--(\ref{zEqn2}) exists globally for $S \ge 0$. The solution $\bX(\tau)$ is recovered from (\ref{renormCoord2}) and \eqref{eq_T_of_S} via $\bX(\tau) = \tilde{\bX}({S}(\tau))$ with $\tilde{\bX}(S) = e^{\tilde{Z}(S)}\tilde{\bY}(S)$.

Since system (\ref{yEqn2}) has the same form (apart from initial conditions) as (\ref{yEqn}) in the previous section, we can use the same terminology of focusing and defocusing attractors given by Definition~\ref{def1}. Recall that the attractors describe the dynamics of spherical components $\tbY(S)$, while the property of being focusing or defocusing characterizes the radial variable $\tilde{R}(S) = e^{\tZ(S)}$. In the focusing case, orbits of the ideal system starting in the basin of attraction converge exponentially to the origin $\tilde{R}(S) \to 0$ as $S \to \infty$, and in the defocusing case they diverge exponentially to infinity.

Clearly, for times $t < t_b$ before the singularity, we can use continuous dependence of the solution $\bx^\nu(t)$ on the parameter $\nu$. Hence, the limit $\nu \to 0$ exists, it is unique and equal to the solution $\bx(t)$ of the original system (\ref{ODE}). By continuity, we extend this statement to the singular point, i.e., for the full time interval $t \in [t_0,t_b]$. The renormalized system (\ref{yEqn2})--(\ref{zEqn2}) facilitates the study of the inviscid limit for times after the blowup, $t >t_b$. Though the role of renormalization is similar to the blowup analysis in Section~\ref{sec3}, the implications are different.

\begin{thm}\label{thm:pastBlowup}
Let $\bx^\nu(t)$ be a solution to the regularized problem \eqref{eqn:epsReg} for the initial condition $\bx_0 = r_0\by_*$ satisfying $\bFF_s(\by_*)=0$ and $F_r(\by_*) < 0$. 
\begin{itemize}
  \item[(a)] For the trapping regularization, the inviscid limit is trivial: $\bx(t) = \lim_{\nu \to 0}\bx^\nu(t) = 0$ for all $t > t_b$. \vspace{1mm}
  \item[(b)] For the expelling regularization, let $\nu_n\to 0$ be a subsequence providing (by Theorem~\ref{convSub}) the uniformly convergent solutions $\bx^{\nu_n}(t) \to \bx(t)$ in a given time interval $[t_b,t_1]$. Let $\by_*^{esc}\in \mathcal{B}(\cA')$ for a defocusing attractor $\cA'$ of system (\ref{yEqn2}) and assume that the solution $\tbY(S)$ satisfies the condition
\begin{align}\label{eq_cond_Y}
\limsup_{S\to \infty} \int_0^S \exp\left[-(1-\alpha)\int_{S'}^SF_r (\tbY(S''))dS''\right]dS'< \infty.
\end{align}
Then $r(t)>0$ and $\by(t) \in \cA'$ for all $t > t_b$. 
\end{itemize}
\end{thm}

 
The conclusions of part (b) of Theorem~\ref{thm:pastBlowup} constitute a severe restriction on the limiting solutions: the spherical component of the solution must belong to the attractor $\mathcal{A}'$ given by the ideal system, independently on a particular form of regularization function $\mathbf{G}$.
 Before proving the above theorem, we provide a Proposition below which gives {an example of concrete condition} on the function $F_r$ that guarantees that property \eqref{eq_cond_Y} holds.

{\begin{prop}\label{prop_cond}
Inequality (\ref{eq_cond_Y}) in Theorem~\ref{thm:pastBlowup} holds if there are constants $c_1 > 0$ and $c_0 \in \mathbb{R}$ such that the inequality
\begin{align}\label{eq_cond_Y3}
\int_{S'}^SF_r (\tbY(S''))dS'' > c_1(S-S')+c_0
\end{align}
is satisfied for any $S > S' \ge 0$.
In particular, this is the case when
$F_r(\bY) > 0$ for all $\bY \in \mathcal{A}'$.
\end{prop}} 

{\begin{proof}
Since $\alpha < 1$, we can use (\ref{eq_cond_Y3}) in the estimate
\begin{align}\label{eq_cond_Y2}
\int_0^S \exp\left[-(1-\alpha)\int_{S'}^SF_r (\tbY(S''))dS''\right]dS' <  
e^{-(1-\alpha)c_0}\int_0^S e^{-c_1(1-\alpha)(S-S')}dS' 
< \frac{e^{-(1-\alpha)c_0}}{(1-\alpha)c_1}.
\end{align}
This implies  (\ref{eq_cond_Y}). In the case, when $F_r(\bY) > 0$ for all $\bY \in \mathcal{A}'$ on the attractor, we can choose $c_1 = \frac{1}{2}\min_{\bY\in\mathcal{A}'}F_r(\bY) > 0$.
Then, for any solution $\tbY(S)$ attracted to $\mathcal{A}'$, there exists an $S_1 < \infty$ such that $F_r(\tbY(S)) > c_1$ for all $S\geq S_1$. Thus, inequality (\ref{eq_cond_Y3}) is satisfied for $S > S' \ge S_1$ with $c_0 = 0$. One can verify that this inequality can be extended to the intervals with $S > S' \ge 0$ and 
\begin{align}\label{eq_cond_Y3pr}
c_0=-c_1S_1-\int_{0}^{S_1} |F_r (\tbY(S''))|dS''.
\end{align} 
\end{proof}}

We now proceed with the proof of Theorem \ref{thm:pastBlowup}.

\begin{proof}[Theorem \ref{thm:pastBlowup}]
(a) From Eq. (\ref{eqRN1}), we express $r^\nu = \nu R$ for $t = t_b+\nu^{1-\alpha}\tau^\nu > t_b+\nu^{1-\alpha}\tau_{ent}$.
In case of trapping regularization, the renormalized solution is bounded, $R(\tau) \le R_b$. With these properties, in the limit $\nu \to 0$ we obtain $r = 0$ for $t > t_b$. 

(b) Consider a subsequence $\lim_{n\to \infty}\nu_n\to 0$ given by Theorem~\ref{convSub}. It provides the uniformly convergent limit $\bx(t) = \lim_{n\to\infty}\bx^{\nu_n}(t)$, which solves equations \eqref{ODE} and \eqref{ODEb}. Let us fix some time $t > t_b$. Since $\alpha < 1$, the sequence of corresponding values of $\tau_n = \tau^{\nu_n}(t)$ given by (\ref{eqRN1}) diverges: $\lim_{n\to\infty} \tau_n = \infty$. The corresponding values of the renormalized time ${S}_n = {S}(\tau_n)$ can be obtained from \eqref{renormCoord2}, or implicitly by inverting (\ref{eq_T_of_S}). Here the value of $\tilde{Z}(S) = \ln \tilde{R}(S) \ge 0$ is bounded from below because $\tilde{R}(S) \ge 1$ for all $S>0$ by the assumptions of the expelling regularization. The upper bound is obtained from Eq.~(\ref{zEqn2}) as $\tilde{Z}(S) \le S M$, where $M:=\max_{\bY\in S^{d-1}}F_r(\bY)$. We must have $M> 0$ for the expelling regularization. These estimates applied to the relation (\ref{eq_T_of_S}) yield 
\be
\label{eq_lim0}
\tau_{esc}+S \le \ttau(S) \le \tau_{esc}+\frac{\exp\left[(1-\alpha)S M\right]-1}{(1-\alpha) M} \quad
\textrm{for} \quad S\geq 0.
\ee
Substituting $S= {S}(\tau)$ into the second inequality of \eqref{eq_lim0}, which yields $\ttau({S}(\tau)) = \tau$, after simple manipulation we have
\be
\label{eq_lim1fst}
{S}(\tau) \geq \frac{\ln\left[ (1-\alpha) M(\tau - \tau_{esc}) +1 \right]}{(1-\alpha)M}\quad 
\textrm{for} \quad\tau\geq \tau_{esc}.
\ee
As we already mentioned, $\lim_{n\to\infty} \tau_n = \infty$. The inequality (\ref{eq_lim1fst}) proves that $\lim_{n\to\infty}{S}_n = \infty$.

From (\ref{eqRN1}) and (\ref{eq_T_of_S}), using the fact that $\tilde{\tau}(S)= \tau^\nu(t)$ when $S= {S}(\tau^\nu(t))$, we express
\be
\label{eq_Prv1}
\nu^{1-\alpha} = 
\left(t-t_b-\nu^{1-\alpha}\tau_{esc}\right)\left(\int_{0}^{{S}(\tau^\nu(t))} e^{(1-\alpha)\tZ(S')} dS'\right)^{-1}.
\ee
Using the relations $r = \nu R = \nu e^Z$ with $Z = \tZ({S}(\tau^{\nu}(t)))$ and substituting $\nu$ from (\ref{eq_Prv1}), we obtain 
\be
\label{eq_Prv2}
r(t) = \left(t-t_b-\nu^{1-\alpha}\tau_{esc}\right)^{\frac{1}{1-\alpha}}\left(\int_{0}^{{S}(\tau^\nu(t))} e^{-(1-\alpha)(\tZ({{S}(\tau^\nu(t))})-\tZ(S'))} dS'\right)^{-\frac{1}{1-\alpha}}.
\ee
In the inviscid limit $\nu_n \to 0$, the first factor in (\ref{eq_Prv2}) tends to $(t-t_b)^{\frac{1}{1-\alpha}} > 0$. Integrating Eq.~(\ref{zEqn2}) as 
\be
\label{eq_Prv3}
\tZ({{S}(\tau^\nu(t))})-\tZ(S') = \int_{S'}^{{{S}(\tau^\nu(t))}}F_r(\tbY(S'')) dS''
\ee
and using (\ref{eq_cond_Y}), we conclude that expression (\ref{eq_Prv2}) provides $r(t) > 0$ for $t > t_b$ in the limit $\nu_n \to 0$.

For the angular variables, we write
\be
\label{eq_lim1}
\by(t)=\lim_{n\to \infty} \by^{\nu_n}(t) = \lim_{n\to \infty} \bY(\tau_n) = \lim_{n\to \infty} \tbY({S}_n).
\ee
Since  $\lim_{n\to\infty}{S}_n = \infty$ and the initial condition $\tbY(0) = \by_*^{esc} \in \mathcal{B}(\cA')$ is assumed to be in the basin of attraction, then $ \lim_{n\to \infty} \tbY({S}_n)\in \cA'$ and we conclude that $\by(t)$ belongs to the attractor $\cA'$.  
\end{proof}

\subsection{Extension to generic initial conditions}\label{sec:generic_IC}

Now let us return to the generic case, when the initial condition in (\ref{ODE}) has the angular part $\by_0$ belonging to the basin of attraction of the fixed point, $\mathcal{B}(\{\by_*\})$, instead of being exactly $\by_*$. By Theorem~\ref{th:EnterAttractor}, the corresponding solution blows up in finite time $t_b$ with the angular part tending to the fixed-point attractor, $\by(t) \to \by_*$, as the time increases to the moment of blowup, $t \to t_b$. For the $\nu$-regularized problem this implies that, for a sufficiently small $\nu$, there is time $t_{ent}^\nu \in [t_0,t_b)$, when the solution enters the regularization region. Furthermore, the value $\by_{ent}^{\nu} = \by(t_{ent}^\nu)$ can be made arbitrarily close to $\by_*$. Therefore, the inviscid limit $\nu \to 0$ corresponds to the problem (\ref{ODErescale}), where $\by_*$ in the initial condition is substituted by an arbitrarily close state $\by_{ent}^{\nu}$ depending on $\nu$.

We already mentioned that the assumptions of Theorem~\ref{thm:pastBlowup} can be expected to be generic, and in the following sections we will prove that this is true in some cases. Hence, under certain non-degeneracy assumptions, it should be possible to extend both statements of Theorem~\ref{thm:pastBlowup} to arbitrary initial conditions with the same asymptotic blowup behavior, i.e., $\by_0 \in \mathbf{B}(\{\by_*\})$. For this purpose, let us enhance the notions of expelling and trapping regularizations. We call the regularization \textit{locally expelling} if all solutions $\bX(t)$ of (\ref{ODErescale}) with the initial conditions $\bX(\tau_{ent}) = \by_{ent}$ in some neighborhood of $\by_*$ are expelled in the sense of Definition~\ref{def2}; additionally, we require that the point $\by_{esc} = \bX(\tau_{esc})$ depends smoothly on $\by_{ent}$. The latter condition is provided naturally by the smooth dependence on initial conditions, if the vector $\bFF(\by_*^{esc})$ is transversal to the unit sphere, where $\by_*^{esc}$ corresponds to the solution starting exactly at $\by_*$.
Similarly, we call the regularization \textit{locally trapping}, if all solutions $\bX(t)$ of (\ref{ODErescale}) with the initial conditions in some neighborhood of $\by_*$ are trapped in the sense of Definition~\ref{def2} with the same bound $R(\tau) \le R_b < \infty$ for $\tau > \tau_{ent}$. Now we can formulate the straightforward extension of Theorem~\ref{thm:pastBlowup} as

\begin{corollary}\label{corollary1}
Let $\bx^\nu(t)$ be a solution to the regularized problem \eqref{eqn:epsReg} for the initial condition $\bx_0 = r_0\by_0$ with $\by_0 \in \mathcal{B}(\{\by_*\})$, where  $\bFF_s(\by_*)=0$ and $F_r(\by_*) < 0$. 
\begin{itemize}
  \item[(a)] For the locally trapping regularization, the inviscid limit is trivial: $\bx(t) = \lim_{\nu \to 0}\bx^\nu(t) = 0$ for all $t > t_b$. \vspace{1mm}
  \item[(b)] For the locally expelling regularization, let $\nu_n\to 0$ be a subsequence providing (by Theorem~\ref{convSub}) the uniformly convergent solutions $\bx^{\nu_n}(t) \to \bx(t)$ in a given time interval $[t_b,t_1]$. Let $\by_*^{esc}$ belong to the basin of attraction $\mathcal{B}(\cA')$ for a defocusing attractor $\cA'$ of system (\ref{yEqn2}). Additionally, we assume that any {solution $\tbY(S)$ starting in some neighborhood of $\by_*^{esc}$ satisfies condition (\ref{eq_cond_Y}).}
Then $r(t)>0$ and $\by(t) \in \cA'$ for all $t > t_b$. 
\end{itemize}
\end{corollary}

Corollary~\ref{corollary1} extends the conclusion of Theorem~\ref{thm:pastBlowup} to an open subset of initial conditions $\bx_0$: the spherical component of the limiting solutions after blowup must belong to the attractor $\mathcal{A}'$ given by the ideal singular system, independently on a particular form of the regularization function $\mathbf{G}$. {This restriction can be very strong, providing a constructive selection rule, as we describe in the next sections for the examples of fixed-point and periodic attractors $\mathcal{A}'$.}

\section{Fixed-point attractor and the unique inviscid limit}\label{sec5_fp}

Let us consider the simplest case when $\cA' = \{\by'_*\}$ is a fixed-point attractor, which is the case when $\bFF_s(\by'_*) = 0$. With the additional condition $F_r(\by'_*) > 0$, we guarantee that the attractor is defocusing and, by Proposition~\ref{prop_cond}, the inequality  (\ref{eq_cond_Y}) holds. The part (b) of Theorem~\ref{thm:pastBlowup} and Corollary~\ref{corollary1} state that the limiting solution for $t > t_b$ must satisfy the original equation (\ref{ODE}) with $\by(t) = \by'_*$ and $r(t) > 0$. Such a solution is unique and can be derived similarly to (\ref{eqFPgen}) in the form
\begin{align} \label{eqFPgenB}
\bx(t)= \left[(1-\alpha)F_r(\by'_*)(t-t_b)\right]^{\frac{1}{1-\alpha}}\by'_*,\quad
t > t_b.
\end{align}
Combining these arguments, we have

\begin{thm}\label{CorFixedPoint}
Let $\bx^\nu(t)$ be a solution to the regularized problem \eqref{eqn:epsReg} for the initial condition $\bx_0 = r_0\by_0$ with $\by_0 \in \mathcal{B}(\{\by_*\})$, where  $\bFF_s(\by_*)=0$ and $F_r(\by_*) < 0$. Assume that the regularization is locally expelling and $\by_*^{esc}$ is in the basin of attraction for $\cA' = \{\by'_*\}$ with $\bFF_s(\by'_*) = 0$ and $F_r(\by'_*) > 0$. 
Then the inviscid limit $\bx(t) = \lim_{\nu \to 0}\bx^\nu(t)$ exists and is given by (\ref{eqFPgenB}).
\end{thm}

We see that the inviscid limit in the case of a fixed-point attractor is unique for all times and it is fully determined by the properties of the ideal system: a specific form of the regularization function $\bG$ has no effect on the limiting solution as far as the generic conditions of Theorem~\ref{CorFixedPoint} are satisfied. We will illustrate this rather counterintuitive property with the following two examples.

\begin{example}\label{ex:fixedpoint}\normalfont
Let us investigate continuation past blowup in the system of Example \ref{Ex1}, assuming {that $\alpha = 1/3$.}  As shown in Fig.~\ref{fig3}(a), there are an infinite number of solutions which start at the singularity. Also, for any time $t_2 > t_b$ there exist solutions which remain at the origin, $\bx(t) = \mathbf{0}$ for $t\in[t_b,t_2)$ and for $t>t_2$ escape the origin along any nontrivial path. We will now see how the regularization provides a specific choice of the solution for $t > t_b$.

\begin{figure}
\centering
\includegraphics[width=0.7\textwidth]{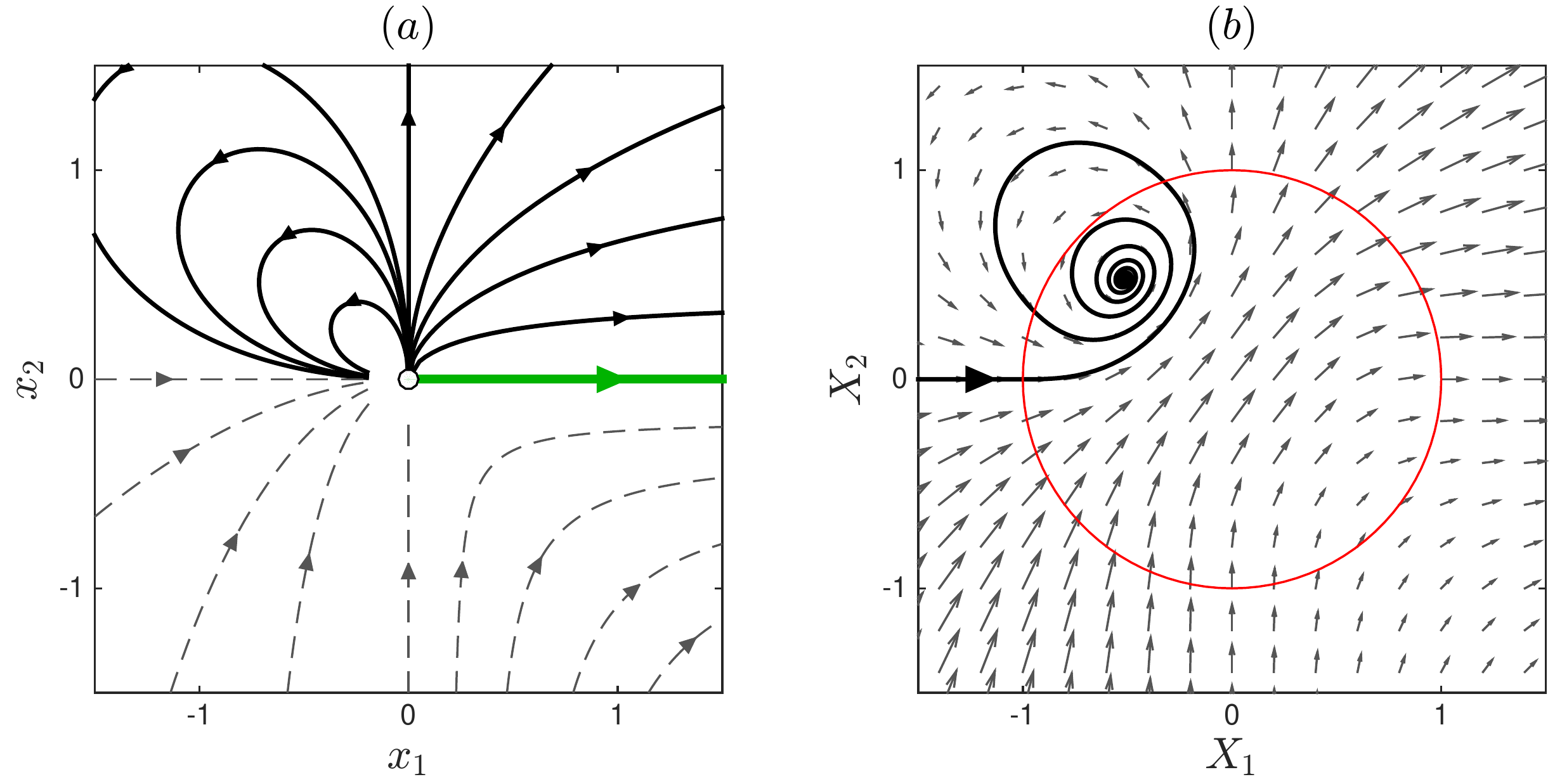}
\hspace{5mm}
\caption{(a) Solid \textbf{black} curves show an infinite number of trajectories starting at the origin for the singular system. The bold \textcolor{ForestGreen}{green} line indicates the unique solution chosen by a generic expelling regularization, see also Fig.~\ref{fig3b}(c). (b) Solution $\bX(\tau)$ of the system with trapping regularization confined to the \textcolor{red}{red} circle  $r \le 1$.}
\label{fig3}
\end{figure}

For the regularized system (\ref{eqn:epsReg}), we define
\begin{align} \label{eq_Gdef}
\bG(\bx) = \xi(r)\bff(\bx)+(1-\xi(r))\bG_0,\quad r = |\bx| \le 1,
\end{align}
where $\bG_0$ is a constant vector and $\xi(r) = 3r^2-2r^3$ smoothly interpolates between $\xi(0) = 0$ and $\xi(1) = 1$. Let us first choose $\bG = (1,\,1.3)$. The corresponding solution $\bX(\tau)$ of the regularized system (\ref{ODErescale}) with $\by_* = (-1,\,0)$ is obtained numerically and presented in Fig.~\ref{fig3}(b). Clearly, the proposed regularization is trapping. By Corollary~\ref{corollary1}, for any initial condition in the left half-plane $x_1 < 0$, the solution blows up in finite time $t_b$ and the inviscid limit provides the trivial solution, $\bx(t) = \mathbf{0}$, for $t \ge t_b$.

As the second choice, we take $\bG = (1,\,-2)$. The corresponding solution $\bX(\tau)$ of the regularized system (\ref{ODErescale}) is shown in Fig.~\ref{fig3b}(a). This regularization is expelling: the solution enters the regularization region $r \le 1$ through the point $\by_*$ and exits forever at $\by_*^{esc}$. The state $\by_*^{esc}$ belongs to the basin of attraction of the fixed point $\by'_* = (1,0)$ (i.e., $\varphi = 0$) of the renormalized system \eqref{yEqn2} with $F_r(\by'_*) = 1 > 0$, see Fig.~\ref{fig1}(b).
By Theorem \ref{CorFixedPoint}, the inviscid limit defines the unique solution 
\begin{align} \label{xsolleave}
\bx(t)=
\begin{pmatrix}\left[(1-\alpha)(t-t_b)\right]^{\frac{1}{1-\alpha} }
\\
0
\end{pmatrix},\quad   t \ge t_b.
\end{align}
Exactly the same solution after blowup is obtained for any initial condition in the half-plane $x_1 < 0$, because these solutions blow up in finite time $t_b$ with the same asymptotic behavior.
This result is confirmed in Fig.~\ref{fig3b}(b) showing the sequence of regularized solutions $\bx^{\nu}(t)$ with $\nu \to 0$. 

\begin{figure}
\centering
\includegraphics[width=0.99\textwidth]{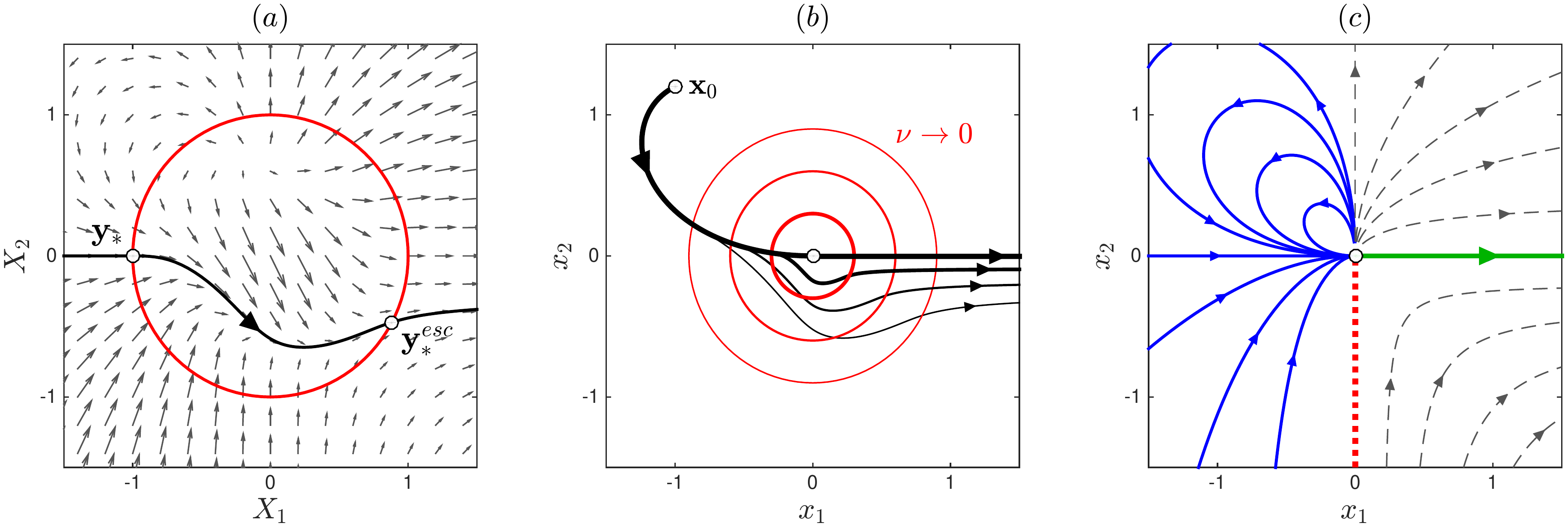}
\caption{Expelling regularization: (a) Solution $\bX(\tau)$ of the rescaled regularized problem; the \textcolor{red}{red} circle indicates the regularization region $R \le 1$. (b) Inviscid limit $\nu \to 0$ of the regularized solutions $\bx^{\nu}(t)$ from the same initial condition. The \textbf{black} curves depict solutions for different values of $\nu$. Together with the corresponding circular regularization regions $r \le \nu$, they are distinguished by the line width. (c) Solutions of the singular system obtained in the inviscid limit. Solid \textcolor{blue}{blue} curves correspond to solutions that blow up in finite time, continued identically past the singularity (\textcolor{ForestGreen}{green} line). Dashed lines correspond to solutions that do not blow up, when initial conditions are taken on the right half-plane. The \red{red} dotted line corresponds to initial conditions that may lead to a non-generic behavior after blowup.}
\label{fig3b}
\end{figure}

We conclude that the regularization provides a unique global-in-time solution to the problem (\ref{ODE})--(\ref{ODEb}) for generic initial condition, see Fig.~\ref{fig3b}(c). In particular, all solutions with $x_1(t_0) < 0$ blow up in finite time and continue past the singularity in exactly the same way (\ref{xsolleave}). The solutions with $x_1(t_0) > 0$, as well as $x_1(t_0) = 0$ and $x_2(t_0) > 0$ do not blowup. Finally, there is a set of initial conditions, $x_1(t_0) = 0$ and $x_2(t_0) \le 0$, which requires a separate analysis. However, this set has zero measure, i.e., the corresponding initial conditions are not generic.
 
It is crucial that, for generic initial conditions, the inviscid limit $\bx(t) = \lim_{\nu \to 0}\bx^\nu(t)$ is defined primarily by the properties of the original singular system, namely, by attractors of the renormalized equations. Thus, the solution does not depend on fine details of the regularization. This means that the solution $\bx(t)$ remains exactly the same under any (sufficiently small) deformation of the regularization function $\bG$. However, very different regularizations (e.g., trapping vs. expelling) may lead to different results. Further choices may appear in case of multiple attractors, as we demonstrate in the next simple example.
\end{example}

\begin{example}\normalfont
{Consider the one-dimension system (\ref{1dmodel}) from the Introduction, which is often used as a prototypical example of non-uniqueness.}
This system has no blowup. However, the equation possesses non-unique solutions starting exactly at the origin. One such solution is $x(t) \equiv 0$ for all {$t \ge t_0$, and there are two extremal solutions (\ref{exactSolb}),}
which leave the origin immediately.  Furthermore, there are a continuum infinity of solutions, which stay at the origin until an arbitrary time {$t_1 \ge t_0$ and peal off as $x(t)=\pm \left(\frac{2}{3}(t-t_1)\right)^{3/2}$ for $t > t_1$.}  

{The renormalized system for (\ref{1dmodel}) is trivial, since $y = x/|x| = \pm 1$ can take only two discrete values for any $x \ne 0$.} Both these values can be seen as fixed-point attractors in the terminology of Theorem~\ref{CorFixedPoint}, with the corresponding limiting solutions given by {(\ref{exactSolb}). Let us modify the problem by replacing \eqref{1dmodel}}  with a $\nu$--regularized dynamics (\ref{eqn:epsReg}) as discussed in \S \ref{sec:nuReg}. Similarly to (\ref{eq_Gdef}), we consider the $\nu$-regularization with 
\be\label{trivialReg}
G(x) = \xi(r)f(x)+[1-\xi(r)]\frac{\sigma+x}{2}, 
\ee
where $r = |x|$ and the sign $\sigma = \pm 1$ defines two different regularizations.
The resulting regularized functions $f^{\nu}(x)$ are shown in Fig.~\ref{figTriv}(a) by the blue lines. This regularization is expelling for the solution with {$x(t_0) = 0$, and the inviscid limit $\nu \to 0$ yields the extreme solution (\ref{exactSolb})} with the same sign $\sigma$. Another example is given by
\be\label{trivialReg2}
G(x) = \xi(r)f(x)+[1-\xi(r)]\frac{1-8x}{6},
\ee
with the function $f^{\nu}(x)$ shown in Fig.~\ref{figTriv}(b) by the red line. In this case, the solution starting at the origin is attracted to a fixed-point located slightly to the right from the origin (trapping regularization). Thus, the inviscid-limit solution is $x(t) \equiv 0$. 

\begin{figure}[t]
\centering
\includegraphics[width=0.7\textwidth]{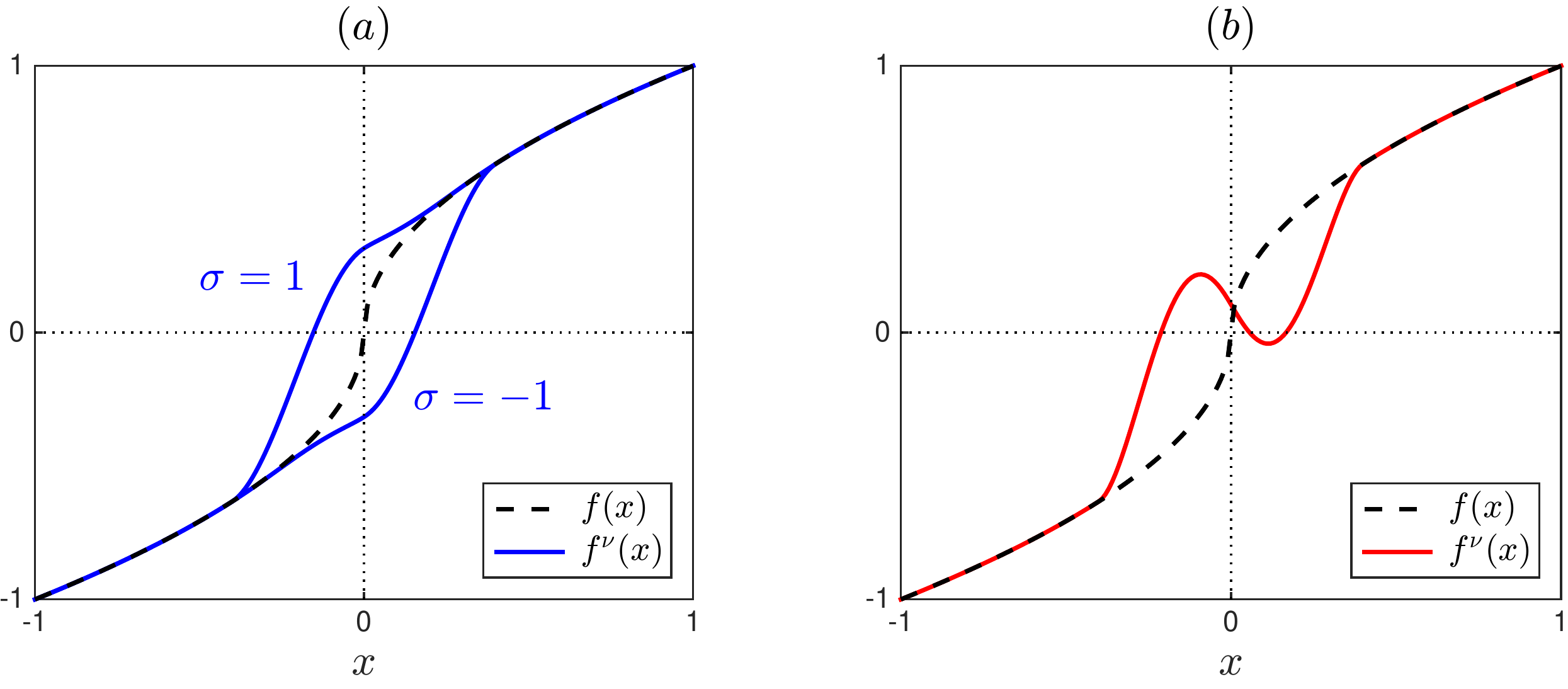}
\caption{ The function {$f(x) = x^{1/3}$} (dashed \textbf{black}) and its $\nu$-regularizations: (a) that expel solutions to the right ($\sigma = 1$) or to the left ($\sigma = -1$) in \textcolor{blue}{blue} and (b) that traps the solutions in \textcolor{red}{red}. The regularization region is $|x| \le \nu$ with $\nu = 0.4$.}
\label{figTriv}
\end{figure}

We see that only three (out of the infinite number of) solutions can be selected by a generic $\nu$-regularization. We remark that the two solutions appear simultaneously for the system with additive-noise regularization:
 \be\label{trivialEXnoisy}
d{x} =  {{\rm sgn}(x)|x|^\alpha dt} +\sqrt{2\kappa} \ dW_t, \quad x(0)=0,
\ee
where $W_t$ is the standard Wiener process and $\alpha\in (0,1)$. In this case, the zero-noise limit ($\kappa \to 0$) selects a non-trivial probability measure (spontaneously stochastic solution), which is distributed symmetrically between the two {extremal solutions (\ref{exactSolb});} see~\cite{attanasio2009zero,weinan2003note,flandoli2013topics}. 
\end{example}

\section{Limit cycle attractor and the non-unique inviscid limit}\label{sec6_lc}

In this section we consider the case  when the post-blowup dynamics in Theorem~\ref{thm:pastBlowup}(b)  is governed by a limit cycle attractor $\mathcal{A}'$. The limit cycle is represented (up to a phase shift in $S$) by a periodic solution $\tbY_p(S)$ of system ${d\tbY}/{dS} = \bFF_s( {\tbY})$:
\be \label{EqLC_P0}
\tbY_p(S) = \tbY_p(S+T),
\ee
with period $T$.  We define the mean value of the radial function $F_r(\by)$ on the limit cycle as

\be \label{EqLC_P0meanAav}
\langle F_r \rangle := 
\frac{I(0,T)}{T},\qquad
I(S_1,S_2) = \int_{S_1}^{S_2} F_r(\tbY_p(S))dS.
\ee
For further use, the integral function is extended for $S_1 > S_2$ as $I(S_1,S_2) = -I(S_2,S_1)$. Note that the function $I(S_1,S_2)-\langle F_r \rangle(S_2-S_1)$ is $T$-periodic with respect to both $S_1$ and $S_2$. Hence,
\be \label{EqBoundI}
\langle F_r \rangle(S_2-S_1)+C_m < I(S_1,S_2) < \langle F_r \rangle(S_2-S_1)+C_M
\ee
for some constants $C_m$ and $C_M$.

First, let us describe a family of solutions of the original problem (\ref{ODE})--(\ref{ODEb}), which start at the singularity and are induced by the limit cycle (\ref{EqLC_P0}). 

\begin{prop}
\label{propP}
Consider a periodic solution (\ref{EqLC_P0}) with a positive mean value, $\langle F_r \rangle > 0$. Then there is a family of solutions of system (\ref{ODE})--(\ref{ODEb}) starting at the singularity $\bx = 0$ at $t = t_b$ and having the form
\begin{align}\label{prop1}
\bx(t) = (t-t_b)^{\frac{1}{1-\alpha}}\bX_p\left(\ln\left[(t-t_b)^{\frac{1}{1-\alpha}}\right]+\zeta\right), 
\end{align}
where $\zeta\in\mathbb{R}$ is a constant parameter and the function $\bX_p$ has period $T\langle F_r \rangle$ with respect to its argument; this function is defined as $\bX_p:=\bx_p\circ \psi^{-1}$, where 
\begin{align} \label{prop1B}
\bx_p(s) &= 
e^{-\varphi(s,s)}\tbY_p(s),\\[3pt] 
\label{prop1C}
\varphi(s_1,s_2) &= \frac{1}{1-\alpha} 
\ln\left[\int_{-\infty}^{0}e^{-(1-\alpha)I(s_1+s',s_2)}ds'\right]
\end{align}
and  $\psi^{-1}:\mathbb{R}\mapsto\mathbb{R}$ is the well defined inverse map of the function  $\psi(s):=\varphi(s,0)$.
\end{prop}

\begin{proof}
Following relations (\ref{eq_t_of_s})--(\ref{zEqn}), solution $\bx(t)$ of (\ref{ODE}) at time $t = \tilde{t}(s)$ for $s\geq 0$ can be written as
\begin{align} \label{prProp0}
\bx(\tilde{t}(s)) = e^{\tz(s)}\tby(s)
\end{align}
where the functions $\tz(s)$, $\tby(s)$ and $\tilde{t}(s)$ satisfy the equations 
\begin{align} \label{prProp1}
\frac{d\tby}{ds} = \bFF_s( {\tby}), \qquad
\frac{d\tz}{ds}= F_r(\tby),  \qquad
\frac{d\tilde{t}}{ds} = e^{(1-\alpha)\tz(s)},
\end{align}
and $s$ is the auxiliary variable. By the assumption, the first equation possesses the periodic solution 
\begin{align} \label{prProp4add}
\tby(s) = \tbY_p(s).
\end{align}
The second equation in (\ref{prProp1}) is integrated as
\begin{align} \label{prProp4}
\tz(s) = -I(s,0)-\zeta,
\end{align}
where $I$ is given in (\ref{EqLC_P0meanAav}) and $\zeta$ is an arbitrary integration constant.

Notice that 
\begin{equation}
\label{eq:onemore1}
\lim_{s\to-\infty}\tz(s) = -\infty, \quad
\lim_{s\to +\infty}\tz(s) = +\infty, 
\end{equation}
because of (\ref{EqBoundI}) with $\langle F_r \rangle > 0$. Hence, $|\tilde{\bx}(s)| = e^{\tz(s)} \to 0$ in the limit $s\to-\infty$, i.e., the solution tends to the singularity. Let us choose the solution of the last equation in (\ref{prProp1}) as
\begin{align} \label{prProp5_ex}
\tilde{t}(s) 
= t_b+\int_{-\infty}^s  e^{(1-\alpha)\tz(s')} ds',
\end{align}
which has the property $\tilde{t}(s) \to t_b$ as $s\to-\infty$. Using (\ref{prProp4}), this yields
\begin{align} \label{prProp5}
\tilde{t}(s) 
= t_b+\int_{-\infty}^se^{-(1-\alpha)\left(I(s',0)+\zeta\right)}ds',
\end{align}
where the integral converges because of (\ref{EqBoundI}) for $\langle F_r \rangle > 0$.
Expression (\ref{prProp5}) can be rewritten after changing the integration variable $s' \mapsto s'+s$ and using (\ref{prop1C}) as
\begin{align} \label{prProp5y}
\tilde{t}(s) 
= t_b+e^{(1-\alpha)(\varphi(s,0)-\zeta)}.
\end{align}
Recalling the notation $\psi(s):=\varphi(s,0)$, we write (\ref{prProp5y}) in the form
\begin{align} \label{prProp5yNext}
\psi(s) = \ln\left[(\tilde{t}(s)-t_b)^{\frac{1}{(1-\alpha)}}\right]+\zeta.
\end{align}

Substituting the expression $I(s+s',0)= I(s+s',s)+ I(s,0)$ into the formula (\ref{prop1C}) for $\varphi(s,0)$ and then expressing $I(s,0)$ from (\ref{prProp4}), one can deduce that
\be
\label{eqNewC9a}
\psi(s) = \varphi(s,0) = \varphi(s,s)-I(s,0) = \varphi(s,s)+\tz(s)+\zeta.
\ee
In view of \eqref{prProp5yNext} and \eqref{eqNewC9a}, we have
\be
\label{eqNewC9b}
\tz(s) = \ln\left[(\tilde{t}(s)-t_b)^{\frac{1}{1-\alpha}}\right]-\varphi(s,s).
\ee
Using (\ref{prProp4add}), (\ref{eqNewC9b}) and the definitions (\ref{prop1B}) and \eqref{prProp0}, we express
\be
\label{eqNewC9c}
\bx(\tilde{t}(s))
= (\tilde{t}(s)-t_b)^{\frac{1}{1-\alpha}}\bx_p(s), \qquad s\in \mathbb{R}.
\ee

The function $t = \tilde{t}(s)$ is monotonically increasing with a positive derivative; see the last equation in (\ref{prProp1}). Due to the properties \eqref{eq:onemore1}, this function maps $\tilde{t}:\mathbb{R}\mapsto (t_b,\infty)$. Thus, the inverse function $s = s(t)$, which maps $(t_b,\infty)\mapsto\mathbb{R}$, is well defined. This allows one to rewrite \eqref{eqNewC9c} in the form
\be
\label{eqNewC9cNext}
\bx(t)
= (t-t_b)^{\frac{1}{1-\alpha}}\bx_p(s(t)), \qquad t\in (t_b,\infty).
\ee
The same properties of $\tilde{t}(s)$ imply that the function $\psi(s)$ {in (\ref{prProp5yNext})} is monotonically increasing with a positive derivative, and it maps $\mathbb{R} \mapsto \mathbb{R}$. Hence, the inverse function exists: $\psi^{-1}:\mathbb{R}\mapsto \mathbb{R}$. Using $\psi^{-1}$ in \eqref{prProp5yNext}  evaluated at $s = s(t)$, we have
\be\label{psiinvIdent}
s(t) = \psi^{-1}\left( \ln\left[(t-t_b)^{\frac{1}{1-\alpha}}\right]+\zeta \right).
\ee
Composing (\ref{eqNewC9cNext}) and \eqref{psiinvIdent}, we have 
\be
\bx(t) 
=  (t-t_b)^{\frac{1}{1-\alpha}}(\bx_p\circ\psi^{-1})\left( \ln\left[(t-t_b)^{\frac{1}{1-\alpha}}\right]+\zeta \right)
\ee
which yields \eqref{prop1} as claimed.

Because of the $T$-periodicity of $\tbY_p(s)$ in the expressions (\ref{EqLC_P0meanAav}), one has 
\be
I(s+T+s',s+T) = I(s+s',s),\quad 
I(s+T+s',0) = I(s+s',0)-\langle F_r \rangle T. 
\ee
Using these formulas in (\ref{prop1C}), one obtains for $\varphi(s,s)$ and $\psi(s) := \varphi(s,0)$: 
\be
\varphi(s,s) = \varphi(s+T,s+T),\quad
\psi(s) = \psi(s+T)-\langle F_r \rangle T. 
\ee
The former equality implies that  $\bx_p(s)$ in (\ref{prop1B}) is $T$--periodic, while the latter yields that the composition $\bX_p=\bx_p\circ \psi^{-1}$ has period $T\langle F_r \rangle$.
\end{proof}

\begin{example}\label{limexample}\normalfont
To illustrate Proposition~\ref{propP} with a simple example, let us consider the system 
	\begin{equation}
	\dot{\bx} = r^{\alpha-1}(x_1-x_2,\,x_1+x_2),
	\label{eq3.1}
	\end{equation}
with $\bx = (x_1,x_2)$ and $\alpha < 1$. The corresponding vector field is shown in Fig.~\ref{fig6}, demonstrating that all solutions emanate from the singularity at the origin. We will find these solutions explicitly assuming that they start from the singularity $\bx(t_b) = \mathbf{0}$ at time $t=t_b$. 
 
\begin{figure}[h]
\centering
\includegraphics[width=0.35\textwidth]{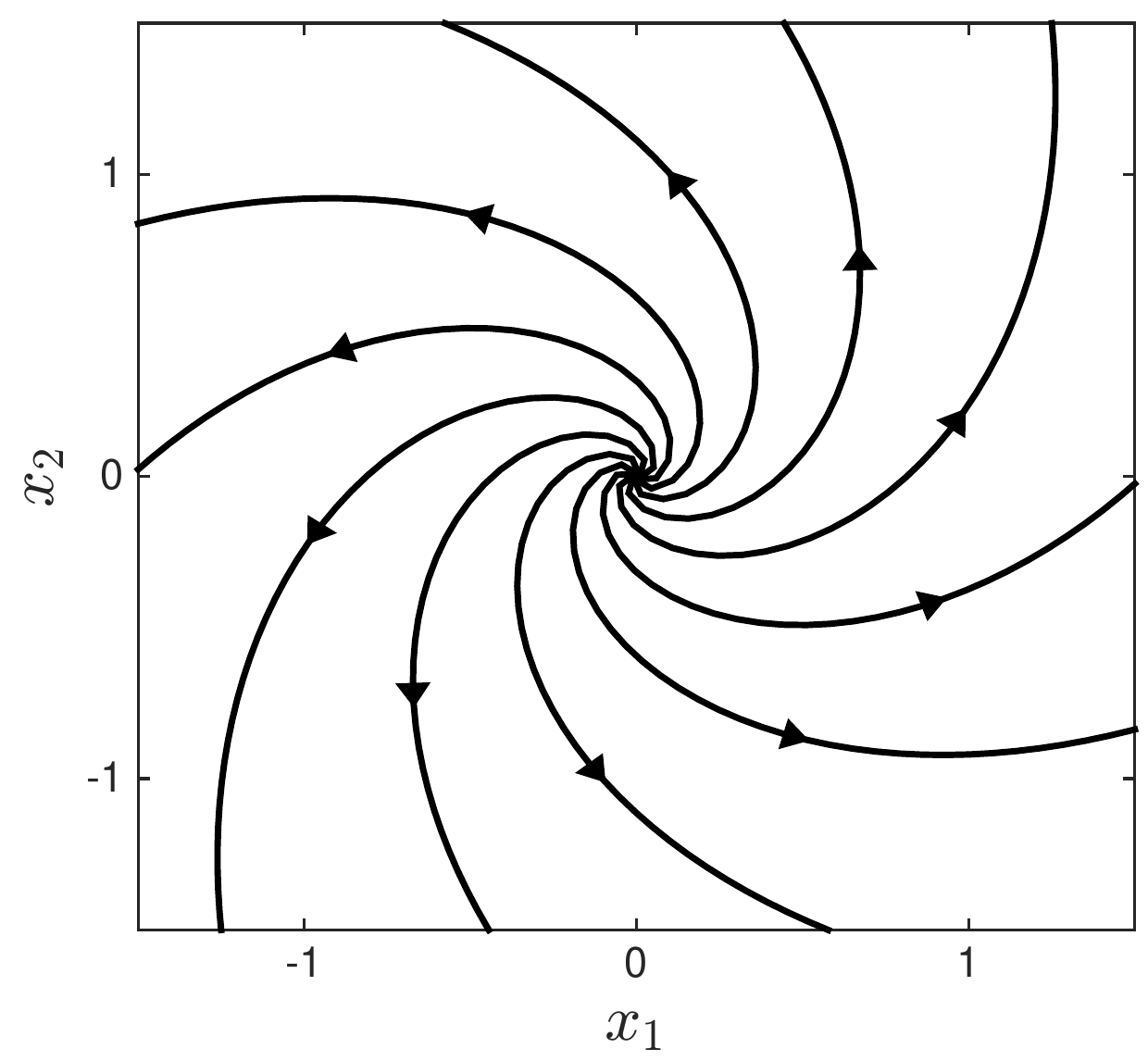}
\caption{Vector field and non-unique trajectories emanating from the origin for Eq.~(\ref{eq3.1}).}
\label{fig6}
\end{figure}

The right-hand side of system (\ref{eq3.1}) can be written as $r^\alpha\bFF(\by)$ with $\by = (y_1,y_2) = \bx/r$ and $\bFF(\by) = (y_1-y_2,y_1+y_2)$. As in Example \ref{Ex1}, it is convenient to work with the angle variable $\varphi\in S^1$ on the circle $\by=(\cos \varphi,\,\sin \varphi)$. Then the radial and circular components of $\bFF$ are given by
\begin{equation}\label{exampleFs}
F_r(\varphi)= F_s(\varphi) = 1.
\end{equation}
For the renormalized system \eqref{phiEqn}, the first equation $d{\tvf}/ds = F_s( {\tvf})$ has the particular solution $\tilde{\varphi}(s) = s$.  In the original variables, this yields the $2\pi$-periodic solution
	\begin{equation}
	 \tbY_p(s)= (\cos s,\ \sin s).
	\label{solnCirc}
	\end{equation}
Since $F_r \equiv 1$, we have $I(s_1,s_2)=s_2-s_1$ from Eq.~\eqref{EqLC_P0meanAav}.  A simple calculation using the definition \eqref{prop1C} yields 
	\begin{equation}
	\label{solnCircB}
	\varphi(s,s) =  -\frac{\ln(1-\alpha)}{1-\alpha},
	\quad
	\psi(s) := \varphi(s,0) =  s-\frac{\ln(1-\alpha)}{1-\alpha}.  		
	\quad
	\psi^{-1}(\xi) =  \xi+\frac{\ln(1-\alpha)}{1-\alpha}.  		
	\end{equation}
Then, formulas \eqref{prop1} and \eqref{prop1B} of Proposition~\ref{propP} yield the explicit solutions
	\begin{align}
	\label{eqEx6end}
	\bx(t) &= \left[(1-\alpha) (t-t_b)\right]^{\frac{1}{1-\alpha}}
	(\cos \xi,\, \sin \xi),\quad
	\xi = \ln\left[(t-t_b)^{\frac{1}{1-\alpha}}\right] +\zeta_1,
	\end{align}
where {$\zeta_1 = \zeta+(1-\alpha)^{-1}\ln(1-\alpha)$} is an arbitrary constant parameter.

We obtained a family of solutions with the same initial condition at the singularity, which depend $2\pi$-periodically on the constant phase {parameter $\zeta_1$.} When $\alpha \in (0,\,1)$, one can also construct solutions that sit at the origin for an arbitrary period of time prior to being shed off following any of the paths \eqref{eqEx6end}. This descries all (non-unique) solutions that start at the singularity at finite time. For system \eqref{eq3.1}, all such solutions are related to the limit cycle in the renormalized equation. This is not the case in general, as we will see in the next example: for systems of higher dimension, $d > 2$, solutions of Proposition~\ref{propP} form a zero-measure subset of all solutions originating from the singular point.
\end{example}

Now let us consider the $\nu$-regularized problem. The results of Theorem \ref{thm:pastBlowup} establish that any solution after blowup obtained by a sequence of expelling regularizations must have angular part that lives on an attractor of the renormalized system. We now prove that, for limit cycle attractors, solutions of Proposition~\ref{propP} are the only possibility arising from ``inviscid limit" of expelling regularizations.  

We assume that all characteristic exponents of the linearized problem near the limit cycle $\tbY_p(S)$ have negative real parts, except for the single vanishing exponent responsible to time-translations, $\tbY_p(S+\delta S)$. In this case the corresponding attractor $\mathcal{A}' = \{\bY = \tbY_p(S): S \in [0,\,T)\}$ is exponentially stable. More specifically, every solution $\tbY(S)$ with the initial conditions from the basin of attraction, $\tbY(0) = \by_*^{esc} \in \mathcal{B}(\mathcal{A}')$, approaches the limit cycle exponentially fast
\be \label{EqLC_P1}
\tbY(S) = \tbY_p(S+S_1)+\boldsymbol{\varepsilon}_Y(S)
\ee
with
\be\label{expDec}
|\boldsymbol{\varepsilon}_Y(S)| < C_Y e^{-\lambda S} \quad \text{for} \quad S \ge 0
\ee
for some constant phase $0 \le S_1 < T$ and $\lambda > 0$, $C_Y>0$; see, e.g.~\cite[p.~254]{hartman2002ordinary}.
Since $\tbY(S)$ tends to the limit cycle, we can define the average value 
\be \label{EqLC_P0mean}
 \langle F_r \rangle = \lim_{S\to\infty}\frac{1}{S}\int_0^S F_r (\tbY(S'))dS'=\frac{1}{T}\int_0^T F_r(\tbY_p(S))dS,
 \ee
which is the same as the average (\ref{EqLC_P0meanAav}) on the periodic attractor $\mathcal{A}'$. Recall that $\langle F_r \rangle \ge 0$ is a necessary condition for the expelling regularization. 

The next theorem characterizes all solutions that can be obtained in the inviscid limit after the blowup for a specific type of initial conditions. The result is a one-parameter family of solutions, depending on the viscous {subsequence. First, we} consider initial conditions of the self-similar blowup, which is governed by a fixed-point attractor $\bFF_s(\by_*)=0$ with the property $F_r(\by_*) < 0$, see Section~\ref{sec3}. Later, we will extend the results to more general initial conditions.

\begin{thm}
\label{theorem_per}
Let $\bx^\nu(t)$ be a solution to the regularized problem \eqref{eqn:epsReg} for the initial condition $\bx_0 = r_0\by_*$, where  $\bFF_s(\by_*)=0$ and $F_r(\by_*) < 0$. Assume that the regularization {is expelling} and $\by_*^{esc}$ is {in the basin of attraction} $\mathcal{B}(\cA')$ for an exponentially stable limit cycle $\cA'$ with the average $\langle F_r \rangle > 0$. 
Then the inviscid limit $\bx(t) = \lim_{n \to \infty}\bx^{\nu_n}(t)$ exists for the sequence (geometric progression)
\be \label{EqLCTn}
\nu_n = e^{-T\langle F_r \rangle n+\chi}
\ee
with an arbitrary fixed $\chi$. After the blowup, for $t > t_b$, the limiting solution coincides with the one given by Proposition~\ref{propP} with $\zeta = c-\chi$ for some regularization-dependent constant $c$. 
\end{thm}

\begin{proof}
As described in Section~\ref{sec4.2}, for the times $t \in [t_0,t_b]$, the inviscid limit is given by the blowup solution (\ref{eqFPgen2})--(\ref{eqFPgen}) of the original system (\ref{ODE})--(\ref{ODEb}). For the times after the blowup, $t > t_b$, the limit can be studied using the relation (\ref{eqRN1}), where $\bX(\tau)$ is the solution of the $\nu$-independent regularized system (\ref{ODErescale}). After leaving the regularization region, $\tau \ge \tau_{esc}$, this system is equivalent to equations (\ref{eq_T_of_S})--(\ref{zEqn2}) providing separately $\tilde{\tau}(S)$ and  $ \tilde{\bX}(S) = e^{\tZ(S)}\tbY(S)$ as functions of the auxiliary variable $S$. Thus, we start by studying the solutions $\tbY(S)$, $\tZ(S)$ and $\tilde{\tau}(S)$, where behavior of $\tbY(S)$ is already described by relation (\ref{EqLC_P1}).

Since the function $F_r(\bY):S^{d-1} \mapsto \mathbb{R}$ is smooth and defined on the sphere, there is a positive constant $c_{\textrm{var}}$ bounding the variation of this function as
\be \label{EqLCpr0}
|F_r(\bY)-F_r(\bY')| < c_{\textrm{var}} |\bY-\bY'|
\ee
for all $\bY,\bY'\in S^{d-1}$. Using this property with the relations (\ref{EqLC_P1}) and (\ref{expDec}), we have
\be \label{EqLCpr0b}
F_r(\tbY(S)) =  F_r(\tbY_p(S+S_1))+\varepsilon_r(S),
\ee
where the exponentially decaying correction term is bounded as
\be\label{expDec2}
|\varepsilon_r(S)| < C_r e^{-\lambda S} \quad \text{for} \quad S \ge 0
\ee
and the positive coefficient $C_r = c_{\textrm{var}}C_Y$.

The solutions $\tZ(S)$ of equation (\ref{zEqn2}) can be written using (\ref{EqLCpr0b}) as
\be \label{EqLCpr5}
\tZ(S) 
= \int_0^S F_r(\tbY(S'))dS'
= \int_0^S F_r(\tbY_p(S'+S_1))dS'+\int_0^S \varepsilon_r(S')dS'.
\ee
Using the integral notation from (\ref{EqLC_P0meanAav}) and introducing the quantities
\be \label{EqLCpr7}
c = I(0,S_1)-\int_0^\infty \varepsilon_r(S')dS',
\qquad
\varepsilon_Z(S) = -\int_S^\infty \varepsilon_r(S')dS',
\ee
we write (\ref{EqLCpr5}) after changing the integration variable $S' \mapsto S'+S_1$ as 
\be \label{EqLCpr8}
\tZ(S) 
= I(0,S+S_1)-c+\varepsilon_Z(S)
= -I(S+S_1,0)-c+\varepsilon_Z(S).
\ee
Notice that the integrals in (\ref{EqLCpr7}) converge because of the bound (\ref{expDec2}). This bound guarantees also that
\be\label{expDec3}
|\varepsilon_Z(S)| < C_Z e^{-\lambda S} \quad \text{for} \quad S \ge 0
\ee
and the positive constant $C_Z = C_r/\lambda$.

The function $\tilde{\tau}(S)$ is given by equation (\ref{eq_T_of_S}), which we write using (\ref{EqLCpr8}) in the form
\be \label{EqLCpr9}
\begin{array}{rcl}
\tilde{\tau}(S) & = & \displaystyle
\tau_{esc}+\int_{0}^S e^{(1-\alpha)\left[-I(S'+S_1,0)-c+\varepsilon_Z(S')\right]} dS'
\\[15pt]
& = & \displaystyle
\tau_{esc}+\int_{-S}^0 e^{(1-\alpha)\left[-I(S'+S+S_1,0)-c+\varepsilon_Z(S+S')\right]} dS',
\end{array}
\ee
where we changed the integration variable $S'\mapsto S+S'$ in the last expression. 
This can be recast as
\be \label{EqLCpr9Bx}
\tilde{\tau}(S) = \tau_{esc}
+e^{-(1-\alpha)c}
\left[1+\varepsilon_1(S)+\varepsilon_2(S)+\varepsilon_3(S)\right]
\int_{-\infty}^0 e^{-(1-\alpha)I(S'+S+S_1,0)} dS',
\ee
where we introduced
\begin{align}
 \label{EqLCpr10a}
\varepsilon_1(S) 
&= 
-
\left(\int_{-\infty}^{0} e^{-(1-\alpha)I(S'+S+S_1,0)} dS'\right)^{-1}
\int_{-\infty}^{-S} e^{-(1-\alpha)I(S'+S+S_1,0)} dS',
\quad
\\
\label{EqLCpr10c}
\varepsilon_2(S) 
&= 
\left(\int_{-\infty}^{0} e^{-(1-\alpha)I(S'+S+S_1,0)} dS'\right)^{-1}
\int_{-S}^{-S/2} e^{-(1-\alpha)I(S'+S+S_1,0)}\left(e^{\varepsilon_Z(S+S')}-1\right) dS',
\\
 \label{EqLCpr10d}
\varepsilon_3(S) 
&= 
\left(\int_{-\infty}^{0} e^{-(1-\alpha)I(S'+S+S_1,0)} dS'\right)^{-1}
\int_{-S/2}^0 e^{-(1-\alpha)I(S'+S+S_1,0)}\left(e^{\varepsilon_Z(S+S')}-1\right) dS'.
\end{align}
Using the function $\varphi$ defined in (\ref{prop1C}), we reduce (\ref{EqLCpr9Bx}) to the form 
\be \label{EqLCpr9B}
\tilde{\tau}(S) = \tau_{esc}
+e^{(1-\alpha)\left[\varphi(S+S_1,0)-c\right]}
\left[1+\varepsilon_{\tau}(S)\right],
\ee
where $\varepsilon_{\tau}(S) = \varepsilon_1(S)+\varepsilon_2(S)+\varepsilon_3(S)$.  

We now show that 
\be \label{proofEx1}
\lim_{S \to \infty }\varepsilon_1(S) = 0,\quad
\lim_{S \to \infty }\varepsilon_2(S) = 0,\quad
\lim_{S \to \infty }\varepsilon_3(S) = 0,
\ee
which implies 
\be \label{proofEx1fin}
\lim_{S \to \infty }\varepsilon_{\tau}(S) = 0.
\ee
From the bounds (\ref{EqBoundI}), for arbitrary $S_a < S_b$, it follows that
\be \label{EqBoundI2}
A_m e^{\beta(S+S_1)}\left(e^{\beta S_b}-e^{\beta S_a}\right)
< \int_{S_a}^{S_b} e^{-(1-\alpha)I(S'+S+S_1,0)}dS' 
< A_M e^{\beta(S+S_1)}\left(e^{\beta S_b}-e^{\beta S_a}\right).
\ee
with $\beta = (1-\alpha)\langle F_r \rangle > 0$ and some positive constants $A_m$ and $A_M$.
Using these inequalities in the definition (\ref{EqLCpr10a}) yields
\be \label{EqLCpr10aN}
\left|\varepsilon_1(S)\right| < A_M e^{-\beta S}/A_m.
\ee
From this estimate, the first limit in (\ref{proofEx1}) easily follows.
Absolute value of the last factor in (\ref{EqLCpr10c}) has the form $\left|e^{\varepsilon_Z(S+S')}-1\right|$ with $-S \le S' \le -S/2$, and it can be bounded by unity for large $S$ using (\ref{expDec3}). Then the second limit in (\ref{proofEx1}) is obtained similarly using (\ref{EqBoundI2}). For $-S/2 \le S' \le 0$ and sufficiently large $S$, one can use the same relation (\ref{expDec3}) and elementary analysis to show 
\be \label{EqExtraNew}
\left|e^{\varepsilon_Z(S+S')}-1\right| < \left|e^{C_Z e^{-\lambda S/2}}-1\right| < 2C_Z e^{-\lambda S/2}.
\ee
Using (\ref{EqBoundI2}) and (\ref{EqExtraNew}) in the expression (\ref{EqLCpr10d}) proves the last limit in (\ref{proofEx1}).

{Now let us return to the regularized solution $\bx^\nu(t)$.}
The functions $\tbY(S)$, $\tZ(S)$ and $\tilde{\tau}(S)$ given by (\ref{EqLC_P1}), (\ref{EqLCpr8}) and (\ref{EqLCpr9B}) define implicitly the function $\bX(\tau)$ with  $\tau = \tilde{\tau}(S)$ and $\bX(\tilde{\tau}(S)) = \tilde{\bX}(S) = e^{\tZ(S)}\tbY(S)$. Then, relations (\ref{eqRN1}) provide the implicit representation for the function $\bx^{\nu}(t)$ with $t=\tilde{t}^{\nu}(S)$ and $\bx^{\nu}(\tilde{t}^{\nu}(S))= \tilde{\bx}^{\nu}(S)$ defined by
\begin{align}
\label{eqNN1}
\tilde{\bx}^\nu(S)=\nu\tilde{\bX}(S) = \nu e^{\tZ(S)}\tbY(S), \quad 
\tilde{t}^{\nu}(S) = t_b+\nu^{(1-\alpha)}\tilde{\tau}(S).
\end{align}
Recall that $\tilde{\tau}(S)$ is the unbounded strictly increasing function with the fixed initial value $\tau(0) = \tau_{esc}$; see (\ref{eq_T_of_S}). Hence,
given a fixed time $t > t_b$ and sufficiently small $\nu > 0$, there exists a unique $S$ satisfying 
\begin{align}\label{EqLCpr12}
t =\tilde{t}^{\nu}(S)= t_b+\nu^{(1-\alpha)}\tilde{\tau}(S).
\end{align}
We denote by $S_n$ the solution corresponding to $\nu_n = e^{-T\langle F_r \rangle n+\chi}$ from (\ref{EqLCTn}). Since $\nu_n \to 0$ as $n \to \infty$, the solution $S_n$ exists and is unique for large $n$. For this solution, we rewrite equation (\ref{EqLCpr12}) using (\ref{EqLCpr9B}) as
\begin{align}\label{proofExA1}
t = t_b+\nu_n^{(1-\alpha)}\tau_{esc}
+e^{(1-\alpha)\left[\varphi(S_n+S_1,0)-T\langle F_r \rangle n-c+\chi\right]}
\left(1+\varepsilon_{\tau}(S_n)\right).
\end{align}
One can see from (\ref{EqLC_P0meanAav}) and (\ref{prop1C}) that 
\begin{equation}\label{proofExA1z}
\begin{array}{rcl}
I(S_n+S_1,0)+T\langle F_r \rangle n &=& I(S_n+S_1-nT,0),\\[3pt]
\varphi(S_n+S_1,0)-T\langle F_r \rangle n &=& \varphi(S_n+S_1-nT,0).
\end{array}
\end{equation}
Therefore, Eq.~(\ref{proofExA1}) takes the form
\begin{align}\label{proofExA1yy}
t = t_b+\nu_n^{(1-\alpha)}\tau_{esc}
+e^{(1-\alpha)\left[\varphi(\xi_n,0)-\zeta\right]}
\left(1+\varepsilon_{\tau}(S_n)\right),
\end{align}
where we defined 
\begin{align}\label{proofExA1yyex}
\xi_n = S_n+S_1-nT, \quad \zeta = c-\chi. 
\end{align}
Expressing $\psi(\xi_n) := \varphi(\xi_n,0)$, we have
\be
\label{proofExA1yyex2}
 \psi(\xi_n) = \frac{1}{1-\alpha}\ln\left[ \frac{t-t_b-\nu_n^{(1-\alpha)}\tau_{esc}}{1+\varepsilon_{\tau}(S_n)}\right]+\zeta.
\ee
Using the inverse map $\psi^{-1}$ defined in Proposition~\ref{propP} and solving \eqref{proofExA1yyex2} for $\xi_n$ yields
\begin{align}\label{proofExA1yyy}
\xi_n = \psi^{-1} \left(
 \frac{1}{1-\alpha}\ln \left[\frac{t-t_b-\nu_n^{(1-\alpha)}\tau_{esc}}{1+\varepsilon_{\tau}(S_n)}\right]+\zeta
 \right).
\end{align}
Recall that, since $\psi$ is continuously differentiable and monotonically increasing,  $\psi^{-1}$ is also continuous. Thus, it is possible to take the limit $n \to \infty$ in the right-hand side with both $\nu_n \to 0$ and $\varepsilon_{\tau}(S_n) \to 0$, which we denote as 
\begin{align}\label{proofExA1yyyy}
s = \lim_{n \to \infty} \xi_n = \psi^{-1} \left(
\ln \left[(t-t_b)^{ \frac{1}{1-\alpha}}\right]+\zeta
 \right).
\end{align}

Finally, from the first relation of (\ref{eqNN1}) with expressions (\ref{EqLC_P1}) and (\ref{EqLCpr8}), we have
\begin{align}\label{proofExA5new}
\tilde{\bx}^{\nu}(S) = 
\nu e^{-I(S+S_1,0)-c+\varepsilon_Z(S)}
\left[\tbY_p(S+S_1)+\boldsymbol{\varepsilon}_Y(S)\right].
\end{align}
Taking this relation for $\nu = \nu_n = e^{-T\langle F_r \rangle n+\chi}$ and $S = S_n$,  yields 
\begin{align}\label{proofExA5}
\tilde{\bx}^{\nu_n}(S_n) = 
e^{-I(S_n+S_1,0)-T\langle F_r \rangle n+\chi-c+\varepsilon_Z(S_n)}
\left[\tbY_p(S_n+S_1)+\boldsymbol{\varepsilon}_Y(S_n)\right].
\end{align}
Adopting notations (\ref{proofExA1yyex}) and using relations (\ref{EqLC_P0}), (\ref{proofExA1z}), we write
\begin{align}\label{proofExA5fin}
\tilde{\bx}^{\nu_n}(S_n) = 
e^{-I(\xi_n,0)-\zeta+\varepsilon_Z(\xi_n-S_1+nT)}
\left[\tbY_p(\xi_n)+\boldsymbol{\varepsilon}_Y(\xi_n-S_1+nT)\right].
\end{align}
We can define the limit 
\begin{align}\label{proofExA5fin}
\tilde{\bx}(s) := \lim_{n \to \infty }\tilde{\bx}^{\nu_n}(S_n) = 
e^{-I(s,0)-\zeta}\tbY_p(s), 
\end{align}
which was computed using (\ref{proofExA1yyyy}) and (\ref{expDec}), (\ref{expDec3}). Recalling the relation (\ref{prProp4}), we see that the relations (\ref{proofExA1yyyy}) and (\ref{proofExA5fin}) provide exactly the solution $\bx(t)$ of Proposition~\ref{propP} in the implicit form. {Therefore,} we proved that the values of the regularized solution $\bx^{\nu_n}(t)$ for given $t > t_b$ (these values are determined by the auxiliary variable $S = S_n$) converge to $\bx(t)$.
\end{proof}

It is now straightforward to extend the result of Theorem~\ref{theorem_per} from the specific initial data $\bx_0 = r_0\by_*$ to an arbitrary point from the basin of attraction of $\by_*$. {As we know from Section~\ref{sec3}, all such initial points lead to the blowup with the same asymptotic form. The next statement describes their continuation past the blowup time.} 

\begin{corollary}
\label{corP}
The statement of Theorem~\ref{theorem_per} remains valid for arbitrary initial condition $\bx_0 = r_0\by_0$, where $\by_0$ belongs to the basin of attraction $\mathcal{B}(\{\by_*\})$ of the fixed-point attractor.
\end{corollary}

\begin{proof}
In Theorem~\ref{theorem_per}, we proved our statement for the specific initial condition with $\by(t_0) = \by_*$. Consider now the case of initial conditions with arbitrary $\by_0 = \mathcal{B}(\{\by_*\})$, i.e., for any initial condition leading to the same asymptotic form of self-similar blowup. In this case, both the time $t_{ent}^\nu$ and the point $\by_{ent}^\nu$, at which the solution enters the regularization region, depend on $\nu$. Let us write the regularized solution $\bx^\nu(t)$ in the form (\ref{eqRN1}), where the function $\bX(\tau)$ satisfies the same $\nu$-independent equations (\ref{ODErescale}) but for the $\nu$-dependent initial condition $\bX(\tau_{ent}^\nu) = \by_{ent}^\nu$.  From the results of Sections~\ref{sec3} it follows that $\by_{ent}^\nu \to \by_*$ in the inviscid limit $\nu \to 0$. Also, the corresponding rescaled time $\tau_{ent}^\nu$ converges to $\tau_{ent}$ given in (\ref{eqRN3}).

Continuous dependence on initial conditions guarantees that, for any fixed $\tau$, the $\nu$-dependent solution $\bX(\tau)$ converges to the analogous $\nu$-independent solution considered in the proof of Theorem~\ref{theorem_per} as $\nu \to 0$. Recall that the solution can be represented as $\bX(\tau) = e^{Z(\tau)}\bY(\tau)$ given implicitly by the renormalized equations (\ref{eq_T_of_S})--(\ref{zEqn2}) outside the regularization region. The convergence for $Z$ and $\bY$ in the inviscid limit is uniform with respect to $\tau$, which follows from the exponential stability of periodic solutions; see, e.g.~\cite[p.~254]{hartman2002ordinary}. One can verify that such uniform convergence is sufficient for extending the proof of Theorem~\ref{theorem_per} to the more general case under consideration.
\end{proof}

The non-uniqueness of post-blowup dynamics described in Theorem~\ref{theorem_per} and Corollary~\ref{corP} was observed after the blowup in the infinite dimensional shell model of turbulence in~\cite{mailybaev2016spontaneous}, where the periodic attractor has the form of a periodic wave in the renormalized system. Below we provide a much simpler illustrative example of finite dimension. 

\begin{example}\label{example4}\normalfont
Consider the system of three equations (\ref{ODE}) with $\bx = (x_1,x_2,x_3)$ and the right-hand side
	\begin{equation}
	\bff(\bx) = r^\alpha\left[\mathbf{a}+\frac{y_3}{2}\mathbf{b}+\left(y_3^2-\frac{1}{4}\right)\mathbf{c}\right],
	\quad \alpha = \frac{1}{3},
	\label{eq_ex4_1}
	\end{equation}
where $r = |\bx|$, $\by = \bx/r$ and
	\begin{equation}
	\mathbf{a} = (-y_2,y_1,0),\quad 
	\mathbf{b} = (y_1,y_2,y_3),\quad 
	\mathbf{c} = \mathbf{a}\times\mathbf{b}.
	\label{eq_ex4_2}
	\end{equation}
In this case the angular and radial parts of the vector field take the form 
	\begin{equation}
	\bFF_s(\by) = \mathbf{a}+\left(y_3^2-\frac{1}{4}\right)\mathbf{c},\quad
	F_r(\by) = \frac{y_3}{2}.
	\label{eq_ex4_3}
	\end{equation}
Phase portrait of the system $d\by/ds = F_s(\by)$ on the sphere $\by \in S^2$ is shown in Fig.~\ref{fig4}. There are two fixed-point solutions (attractor and repeller) and two periodic solutions (attractor and repeller), which confine the qualitative behavior of all other solutions. With the color in Fig.~\ref{fig4} we indicate the regions with $F_r < 0$ (blue) and $F_r > 0$ (red).

\begin{figure}
\centering
\includegraphics[width=0.85\textwidth]{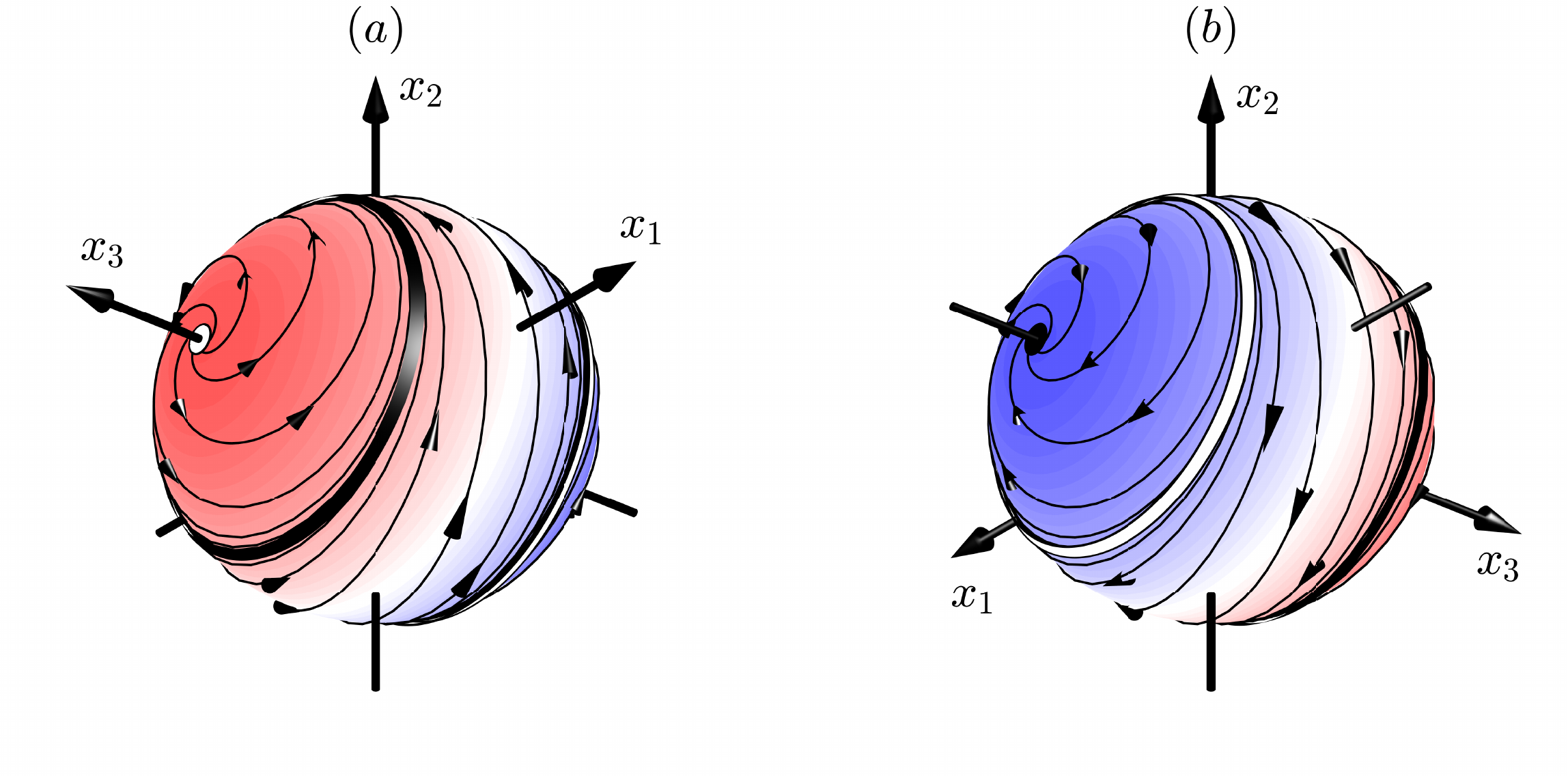}
\caption{Vector field $\bFF_s(\by)$ on the sphere $\by \in S^2$ from two different points of view. The color indicates the sign and magnitude of the axial component $F_r(\by)$ with \textcolor{blue}{blue} corresponding to $F_r < 0$ and \textcolor{red}{red}  to $F_r > 0$. (a) Fixed-point repeller (white dot) and periodic attractor (black strip). (b) Fixed-point attractor (black dot) and periodic repeller (white strip). }
\label{fig4}
\end{figure}

Let us consider the fixed-point attractor $\by_* = (0,\,0,-1)$, show as a black dot in Fig.~\ref{fig4}(b). All solutions with $\by_0 \in \mathcal{B}(\{\by_*\})$ blow up in finite time. The basin of attraction $\mathcal{B}(\{\by_*\})$ {is bounded by} the unstable limit cycle (white strip) in Fig.~\ref{fig4}(b). To define solutions after the blowup, let us consider the regularization (\ref{eq_Gdef}) with $\mathbf{G}_0 = (0,\,0.1,\,1)$. The corresponding solution $\bX(\tau)$ of the regularized system (\ref{ODErescale}) is shown by the bold line in Fig.~\ref{fig8n}(a). The regularization is expelling: the solution $\bX(\tau)$ leaves the unit ball at the point $\by_*^{esc}$. This point belongs to the basin of attraction $\mathcal{B}(\mathcal{A}')$, where $\mathcal{A}'$ is the stable limit cycle shown with the black strip in Fig.~\ref{fig4}(a).

\begin{figure}
\centering
\includegraphics[width=0.76\textwidth]{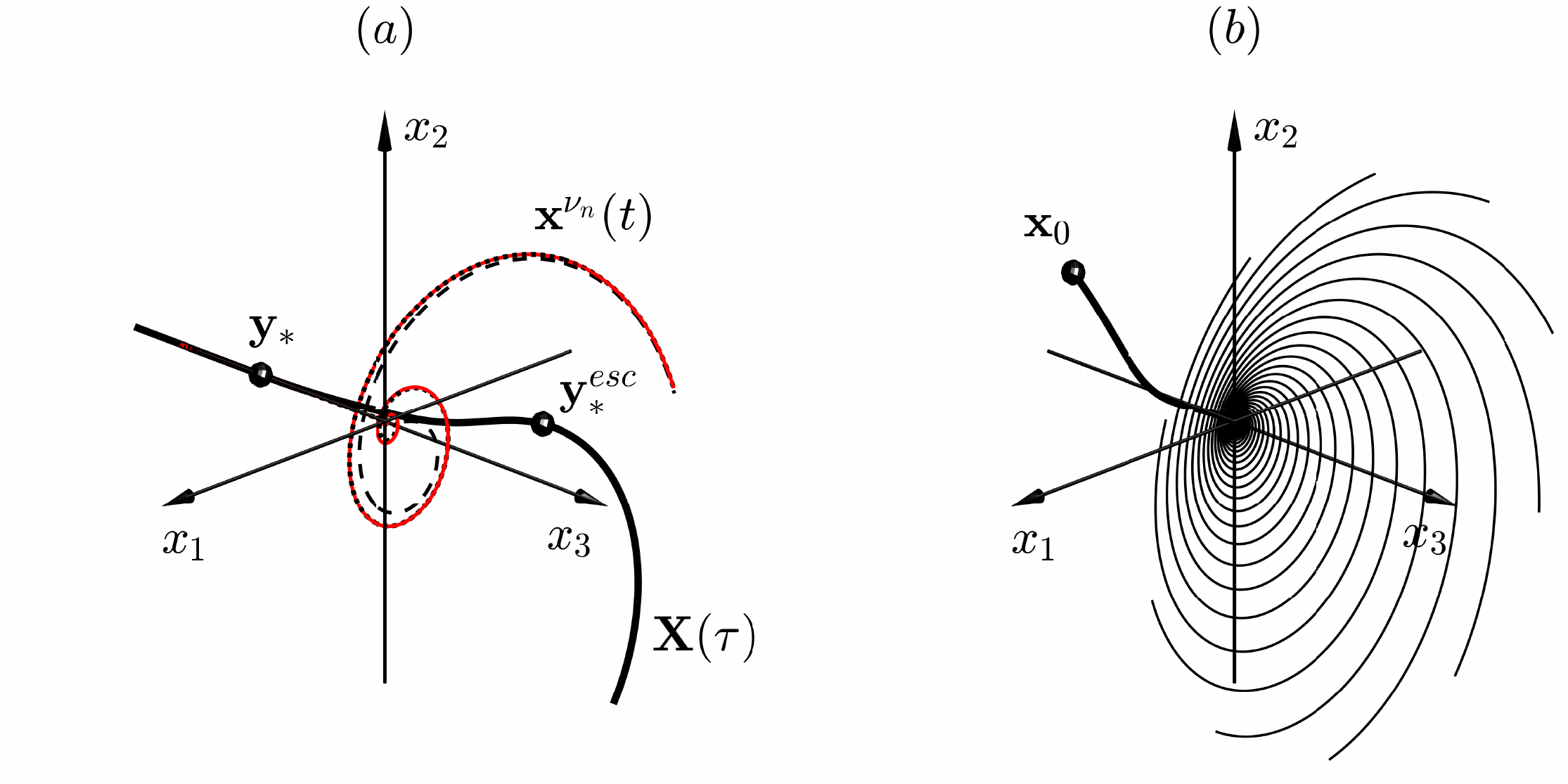}
\caption{(a) Solid \textbf{bold}  line shows the solution $\bX(\tau)$ of the rescaled regularized problem (\ref{ODErescale}). This solution enters the unit ball at $\by_*$ and exists at $\by_*^{esc}$. Regularized solutions $\bx = \bx^{\nu}(t)$ with the same initial condition are shown for $\nu = \nu_n$ from (\ref{EqLCTn}) with $\chi = 0$ and $n = 1$ (dashed line), $n = 2$ (dotted line) and $n = 3$ (\textcolor{red}{red} solid line). These solutions converge to the solution of the singular problem as $n \to \infty$. (b) Different inviscid-limit solutions with the same initial condition $\bx_0$ obtained for the subsequences (\ref{EqLCTn}) with $\chi/T = 0,\,0.1,\ldots,0.9$.}
\label{fig8n}
\end{figure}

By Theorem~\ref{theorem_per} and Corollary~\ref{corP}, solutions in the inviscid limit are given implicitly by expressions (\ref{prop1}) obtained by solving equations (\ref{prProp1}) on the limit cycle attractor. In our case, integration of the first equation in (\ref{prProp1}) with the first expression of (\ref{eq_ex4_3}) yields
	\begin{equation}
	\tbY_p(s) = \left(\frac{\sqrt{3}}{2}\cos s,\,\frac{\sqrt{3}}{2}\sin s,\,\frac{1}{2}\right),
	\label{eq_ex4_4}
	\end{equation}
where the period $T = 2\pi$. The respective average value, $\langle F_r \rangle =  T^{-1}\int_0^T F_r(\tbY_p(s))ds = 1/4$. 
As in Example~\ref{limexample}, from  (\ref{prop1}) we obtain solutions of the form
	\begin{equation}
	\bx(t) = \left[\frac{(1-\alpha)(t-t_b)}{4}\right]^{\frac{1}{1-\alpha}}
	\left(\frac{\sqrt{3}}{2}
	\cos s,\,\frac{\sqrt{3}}{2}\sin s,\,\frac{1}{2}\right),
	\quad
	s = 4\ln\left[(t-t_b)^{\frac{1}{1-\alpha}}\right]+\zeta,
	\label{eq_ex4_6}
	\end{equation}
where $\zeta$ is an arbitrary constant parameter.

The relation (\ref{EqLCTn}) of Theorem~\ref{theorem_per} is verified in Fig.~\ref{fig8n}(a), which presents the solutions computed numerically for $n = 1,2,3$ and $\chi = 0$. These solutions converge to the solution (\ref{eq_ex4_6}) for a specific value of $\zeta$.
A family of solutions (\ref{eq_ex4_6}) for different $\zeta \in [0,\,T)$ span the conical surface as shown in Fig.~\ref{fig8n}(b). Only these solutions are selected in the inviscid limit of the regularization $\bG(\bx)$ under consideration, as well as of regularizations obtained by any (not too large) deformation of $\bG(\bx)$. We stress that solutions (\ref{eq_ex4_6}) represent only a small (zero-measure) subset of all solutions starting at the singularity $\bx = 0$. In fact, one can show that all solutions inside the cone in Fig.~\ref{fig8n}(b) also originate at the singularity.

We conclude that the inviscid limit in this example is non-unique (depends on the particular subsequence $\nu_n \to 0$) for times after the blowup. However, all possible inviscid limits are restricted to a very small (zero measure) subset of all possible solutions. Once again, remarkably,  this subset is selected by the ideal system, i.e., it is not sensitive to the details of the regularization procedure. 

\end{example}

 \section{Discussion}

In the present work, we studied a class of singular ordinary differential equations with an isolated non-Lipschitz point, when a continuation of solutions past singularity (termed as blowup) is infinitely non-unique. We showed that solutions chosen by a generic ``viscous'' regularization procedure, which first smooths the vector field in $\nu$-vicinity of a singular point and then sends $\nu \to 0$,  remain highly constrained by the underlying {singular} equation and these constraints are (nearly) independent of regularization procedure. Such constraints are obtained from the solution-dependent renormalization, which maps the pre-blowup and post-blowup dynamics into two different infinite evolutions in the new phase-time variables. This describes {the asymptotic form of blowup} by the attractor of the first evolution, and the post-blowup continuation by the attractor in the second ``new life'' evolution. {The viscous regularization acts as a bridge between these two infinitely long ``lives'' of the renormalized solution.}

The restrictions imposed in this way {on a selected non-unique solution} depend crucially on the type of attractors. For the pre-blowup dynamics, the fixed-point attractors describe an asymptotically self-similar power-law dynamics closely resembling {self-similar} finite-time singularities in partial differential equations~\cite{eggers2009role}. As for more sophisticated attractors, we can mention {analogies with the} chaotic Belinsky-Khalatnikov-Lifshitz (BKL) singularity in general relativity~\cite{belinskii1970oscillatory,choquet2009general} {or chaotic blowup in turbulence models~\cite{mailybaev2012c,de2017chaotic,campolina2018chaotic}.} Attractors play even more decisive role for post-blowup dynamics, because they describe the exact dynamics rather that its asymptotic form. It is appealing to compare the case of a fixed-point attractor, when the unique solution is selected, with shock wave formation in conservation laws. In both cases the viscous regularization chooses a specific unique solution, which is not sensitive (within certain limits) to the details of this regularization. {For example,} the same shock in the Burgers equation results from the limit of vanishing viscosity or hyper-viscosity. 
{A similar renormalization procedure on both sides of blowup can be formally introduced for the Burgers equation, where the attractors take the form of traveling waves propagating in log-log space-time coordinates~\cite{eggers2009role,mailybaev2016spontaneously}.}

The case of a periodic attractor offers non-unique choices within a one-parameter family, resulting from different geometric subsequences of vanishing viscous parameters $\nu$. Though we are not aware of real-world physical phenomena that have this property, such periodic non-uniqueness was readily observed in a popular infinite-dimensional model of turbulence (shell model)~\cite{mailybaev2016spontaneous}. We remark that the same shell model demonstrates a different regime, when the attractor is chaotic~\cite{mailybaev2016spontaneously}. In this case the solution is expected to be spontaneously stochastic (continued as a probability measure), in the same spirit as in the Kraichnan model \cite{bernard1998slow}, where the source of randomness is introduced and removed through the regularization procedure. We conjecture that a similar behavior can be found for singular systems studied in this paper, which is a topic of ongoing research.

 \section*{Acknowledgements}
\noindent The authors are grateful to Enrique Ramiro Pujals for very useful discussions.
Research of TD is supported by NSF-DMS grant 1703997, and AM  by CNPq grant No. 302351/2015-9.

\bibliographystyle{siam} 
\bibliography{refs}

\end{document}